\documentclass[12pt, english]{article}

\usepackage{babel}
\usepackage{amsmath}
\usepackage{amssymb}
\usepackage{amsthm}
\usepackage{tikz}
\usepackage{pgfplots}
\usepackage{xcolor}
\usepackage{url}

\theoremstyle{plain}
\newtheorem{Thm}{Theorem}
\newtheorem{Prop}[Thm]{Proposition}

\newtheorem{Lemma}[Thm]{Lemma}
\newtheorem{Conjecture}[Thm]{Conjecture}

\theoremstyle{definition}

\newtheorem{Def*}{Definition}

\begin{document}

\title{Moment generating functions in combinatorial optimization: Bipartite matching}
\author{Johan Wästlund}

\date{}

\maketitle

\begin{abstract}
In a random model of minimum cost bipartite matching based on exponentially distributed edge costs, we show that the distribution of the cost of the optimal solution can be computed efficiently. The distribution is represented by its moment generating function, which in this model is always a rational function. The complex zeros of this function are of interest as the lack of zeros near the origin indicates a certain regularity of the distribution. We propose a conjecture according to which these moment generating functions never have complex zeros of smaller modulus than their first pole. For minimum cost perfect matching, also known as the assignment problem, such a zero-free disk would imply a Gaussian scaling limit. 
\end{abstract}

\section{Introduction} \label{S:intro}
Optimization problems on graphs with random edge costs have been studied from various perspectives in mathematics, physics, computer science and machine learning. In addition to their interest as ``large'' random processes, they have been used for testing and evaluating algorithms and as toy models of networks and physical systems. 

In a graph with nonnegative costs assigned to the edges, we typically want to minimize the total cost of a configuration of edges meeting certain connectivity criteria. Some archetypal problems of this kind are minimum matching, the traveling salesman problem, minimum spanning tree, minimum cover, and shortest path. We intend to study all of these in following papers, but here we focus on bipartite matching.

\subsection{Minimum cost matching}

In the basic form of our model, the $n^2$ edges of the complete bipartite graph $K_{n, n}$ are assigned independent costs from mean 1 exponential distribution. A \emph{perfect matching} is a set of $n$ edges connecting each vertex to exactly one vertex on the opposite side, thus constituting a pairing of the vertices. 
Under this constraint we want to minimize the sum of the costs of the $n$ edges, and we let $C_n$ be this minimum total cost. This optimization problem is also called the \emph{assignment problem}.
Figure~\ref{F:matchings} shows the feasible solutions in the case $n=3$. 

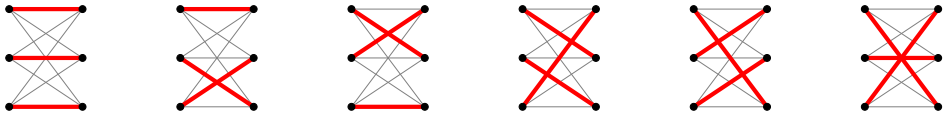
\begin{figure}[h]
\begin{center}
\begin{tikzpicture} [scale=0.65]

\foreach \x in {0,1,2}
  \foreach \y in {0,1,2}
    \draw[gray, ultra thin] (0,\x)--(1.5,\y);   
\draw[ultra thick, red]  (0,0)--(1.5,0);
\draw[ultra thick, red]  (0,1)--(1.5,1);
\draw[ultra thick, red]  (0,2)--(1.5,2);    
\foreach \x in {0,1.5}
  \foreach \y in {0,1,2}
    \filldraw (\x, \y) circle (2pt);
    
\begin{scope}[xshift=3.5cm] 
\foreach \x in {0,1,2}
  \foreach \y in {0,1,2}
    \draw[gray, ultra thin] (0,\x)--(1.5,\y);   
\draw[ultra thick, red]  (0,0)--(1.5,1);
\draw[ultra thick, red]  (0,1)--(1.5,0);
\draw[ultra thick, red]  (0,2)--(1.5,2);    
\foreach \x in {0,1.5}
  \foreach \y in {0,1,2}
    \filldraw (\x, \y) circle (2pt);
 \end{scope}
 
 \begin{scope}[xshift=7cm] 
\foreach \x in {0,1,2}
  \foreach \y in {0,1,2}
    \draw[gray, ultra thin] (0,\x)--(1.5,\y);   
\draw[ultra thick, red]  (0,0)--(1.5,0);
\draw[ultra thick, red]  (0,1)--(1.5,2);
\draw[ultra thick, red]  (0,2)--(1.5,1);    
\foreach \x in {0,1.5}
  \foreach \y in {0,1,2}
    \filldraw (\x, \y) circle (2pt);
 \end{scope}

\begin{scope}[xshift=10.5cm] 
\foreach \x in {0,1,2}
  \foreach \y in {0,1,2}
    \draw[gray, ultra thin] (0,\x)--(1.5,\y);   
\draw[ultra thick, red]  (0,0)--(1.5,2);
\draw[ultra thick, red]  (0,1)--(1.5,0);
\draw[ultra thick, red]  (0,2)--(1.5,1);    
\foreach \x in {0,1.5}
  \foreach \y in {0,1,2}
    \filldraw (\x, \y) circle (2pt);
 \end{scope}

\begin{scope}[xshift=14cm] 
\foreach \x in {0,1,2}
  \foreach \y in {0,1,2}
    \draw[gray, ultra thin] (0,\x)--(1.5,\y);   
\draw[ultra thick, red]  (0,0)--(1.5,1);
\draw[ultra thick, red]  (0,1)--(1.5,2);
\draw[ultra thick, red]  (0,2)--(1.5,0);    
\foreach \x in {0,1.5}
  \foreach \y in {0,1,2}
    \filldraw (\x, \y) circle (2pt);
 \end{scope}

\begin{scope}[xshift=17.5cm] 
\foreach \x in {0,1,2}
  \foreach \y in {0,1,2}
    \draw[gray, ultra thin] (0,\x)--(1.5,\y);   
\draw[ultra thick, red]  (0,0)--(1.5,2);
\draw[ultra thick, red]  (0,1)--(1.5,1);
\draw[ultra thick, red]  (0,2)--(1.5,0);    
\foreach \x in {0,1.5}
  \foreach \y in {0,1,2}
    \filldraw (\x, \y) circle (2pt);
 \end{scope}
\end{tikzpicture}
\end{center}
\caption{The $6=3!$ perfect matchings on the complete bipartite graph $K_{3,3}$. The $9=3^2$ edges are assigned independent mean 1 exponential costs, and the minimum total cost $C_3$ of a perfect matching has expectation $1+1/4+1/9=49/36$ according to equation \eqref{parisi}. We show in the following how to characterize the distribution of $C_n$ completely.} 
\label{F:matchings}

\end{figure}

As $n\to\infty$, the random variable $C_n$ converges in mean and in probability to the constant $\pi^2/6$.
\begin{equation} \label{zeta2} 
C_n \overset{\rm p}\to \zeta(2) = \frac{\pi^2}{6} \approx 1.6449.
\end{equation}
This limit was derived non-rigorously by Marc Mézard and Giorgio Parisi in the 1980's \cite{MP85, MP87, P86}, and proved by David Aldous in 2001 \cite{A92, A01}. Other proofs have since been obtained, for instance in \cite{W08easy}.

The fluctuations of $C_n$ are of order $1/\sqrt{n}$. A bound exceeding the true order only by logarithmic factors was given by Michel Talagrand \cite{T95} in his work on isoperimetric inequalities. In \cite{W10} we showed that
\begin{equation} \label{varLimit} 
\operatorname{var}(C_n) \sim \frac{4\zeta(2)-4\zeta(3)}{n}.
\end{equation}
The constant $4\zeta(2) - 4\zeta(3) \approx 1.7715$ was also derived for a related model by Enrico M.~Malatesta, Giorgio Parisi and Gabriele Sicuro \cite{MPS19}.

\subsection{Finite $n$ results for exponential edge costs} \label{S:exact}
The limits \eqref{zeta2} and \eqref{varLimit} do not depend crucially on the edge costs having exponential distribution. They would remain the same with another nonnegative distribution with density~1 at zero, although the modes of convergence may become weaker. This observation as well as the large $n$ scaling of $C_n$ are explained by the fact that most of the time only extremely cheap edges participate in the minimum cost perfect matching. 

In several papers, the edge costs have been assumed to be uniform on the interval $[0,1]$ with the understanding that the specific distribution is not important.
But in the short note \cite{P98}, Parisi suggested that the expected cost of bipartite perfect matching on $K_{n,n}$ is
\begin{equation} \label{parisi} \mathbb{E}(C_n) = \sum_{i=1}^n\frac{1}{i^2} = 1 + \frac14 + \frac19 + \dots + \frac{1}{n^2},\end{equation}
provided the edge costs are taken from mean 1 exponential distribution. This equation, obviously consistent with the $\zeta(2)$ limit \eqref{zeta2}, was proved simultaneously and independently by Chandra Nair, Balaji Prabhakar and Mayank Sharma \cite{NPS05} and Svante Linusson and myself \cite{LW04}. 

It seems that before Parisi's note, nobody expected an optimization problem as complicated as minimum matching to have such a simple answer in a probabilistic setting. But the formula \eqref{parisi} opened the possibility of establishing results like the $\zeta(2)$ limit (which at the time was not yet proved) through inductive combinatorial arguments based on exact identities for finite $n$. 
 This approach was taken in \cite{AS02, BCR02, CS99, HW10, LW04, NPS05, W08complete, W10}, and here
we continue the exploration of finite $n$ identities by showing that the distribution of $C_n$ (not only its expectation) can be characterized with a relatively simple computation. Figure~\ref{F:densities}  shows the density functions of $C_n$ for $n=1,\dots, 25$. We explain in Section~\ref{S:bipartiteMatching} how these functions were computed.  

\begin{figure}
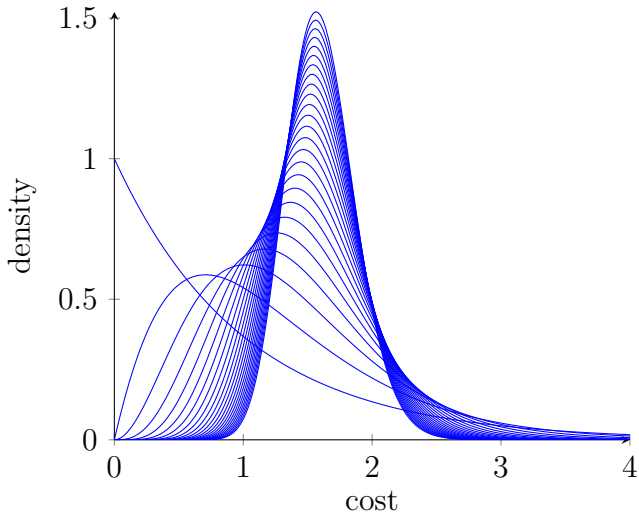

\begin{center}

\caption{The density functions $f_n$ of $C_n$ for $n=1,\dots, 25$. We recognize $f_1$ as the one with $f_1(0)=1$, and $f_2$ by $f'_2(0)>0$. Then $f_3,\dots, f_{25}$ follow in order of increasing density around $\pi^2/6\approx 1.64$.
}
\label{F:densities}
\end{center}
\end{figure}

\subsection{Conjectured Gaussian limit} \label{S:prelGaussian}
We make some general remarks about the conjectured large $n$ behavior of $C_n$. In Section~\ref{S:cumulants} we continue the discussion in the light of our current results.

In the $n$ by $n$ complete bipartite graph $K_{n,n}$ each vertex has $n$ edges, of which exactly one will participate in the minimum cost perfect matching. We expect that most often the participating edge will be one of the cheaper. The cost of the very cheapest of the $n$ edges at a given vertex is exponential of mean $1/n$, and the $\zeta(2)$ limit shows that the average cost of an edge in the minimum perfect matching is only about $1.64$ times this minimum (what Talagrand called the ``cost of monogamy'' in \cite{T00}). Moreover, although not easy to justify even heuristically, it seems that the costs of matching vertices that are distant in the geometry induced by the edge costs are close to independent. We therefore expect $C_n$ to behave qualitatively like a sum of $n$ independent random variables with mean of order $1/n$. 

A natural question is whether, after proper rescaling, $C_n$ obeys a central limit theorem in the sense of converging to a Gaussian distribution. This was at least implicitly conjectured already in \cite{MP85, MP86b, MP87}, and Figure~\ref{F:densities} is visually consistent with a Gaussian limit. 
In view of \eqref{zeta2} and \eqref{varLimit}, we conjecture more precisely that
\begin{equation} \label{rescaled}
\tilde{C}_n = \sqrt{n}\cdot (C_n - \pi^2/6) \overset{\mathcal{D}}\longrightarrow \mathcal{N}(0,4\zeta(2)-4\zeta(3)),
\end{equation}
where $\tilde{C}_n$ is the rescaled cost. Notice that the difference between \eqref{parisi} and its limit $\pi^2/6$ is of order $1/n$ which is smaller than the fluctuations of $C_n$. Therefore the conjecture will be the same whether we put $\mathbb{E}(C_n)$ or its limit $\pi^2/6$ in \eqref{rescaled}.

To the best of my knowledge no proof of \eqref{rescaled} is known, but in \cite{Cao19} Sky Cao has shown, using Stein's method \cite{Chatterjee14, Chatterjee17}, that the cost of a certain ``diluted'' matching problem is asymptotically Gaussian. 

\subsection{Cumulants}
Denoting the moment generating function of $C_n$ by
\[ F_n(t) = \mathbb{E}\big[\exp(t\cdot C_n)\big], \]
the series expansion of $F_n$ around $t=0$ gives the ``raw'' moments $\mu_p(C_n) = \mathbb{E}\big[C_n^p\big]$ by
\[ F_n(t) = 1 + \mu_1(C_n) \cdot t + \frac{\mu_2(C_n) \cdot t^2}{2!} + \frac{\mu_3(C_n) \cdot t^3}{3!} + \dots\]
As we discuss in Section~\ref{S:observations}, this series converges for $\left| t \right| < n$.
The cumulants $\kappa_p(C_n)$ of $C_n$ are defined similarly by the series
\[ \log F_n(t) = \kappa_1(C_n) \cdot t + \frac{\kappa_2(C_n) \cdot t^2}{2} + \frac{\kappa_3(C_n) \cdot t^3}{3!} + \dots \]
The first two cumulants $\kappa_1$ and $\kappa_2$ are the mean and variance respectively. 
Letting $\tilde{F}_n(t)$ denote the moment generating function of the rescaled cost $\tilde{C}_n = \sqrt{n}\cdot (C_n - \pi^2/6)$ from equation \eqref{rescaled}, it follows that
\begin{equation} \tilde{F}_n(t) = \mathbb{E}\big[\exp(t\sqrt{n}\cdot C_n)\cdot\exp(-t\sqrt{n}\cdot \pi^2/6)\big] 
= e^{-t\sqrt{n}\pi^2/6}\cdot F_n(t\sqrt{n}),
\end{equation}
and
\begin{multline}
\log \tilde{F}_n(t) = \\ \left(\kappa_1(C_n)-\pi^2/6\right)\sqrt{n}\cdot t 
+ \frac{\kappa_2(C_n)\cdot n \cdot t^2}{2!} + \frac{\kappa_3(C_n)\cdot n^{3/2} \cdot t^3}{3!} + \dots 
\end{multline}
It is known (see for instance \cite[Section~2]{Tao} and its references) that a sufficient condition for a sequence of random variables to converge to a Gaussian distribution is that as $n\to \infty$, its first two cumulants converge (at all), and the third and following cumulants tend to zero. 
From \eqref{zeta2} and \eqref{varLimit} we have 
\[ \kappa_1(\tilde{C}_n) = \sqrt{n}\cdot (\kappa_1(C_n)-\pi^2/6) = O(1/\sqrt{n}) \to 0,\] and 
\[ \kappa_2(\tilde{C}_n) = n \cdot \kappa_2(C_n) \to 4\zeta(2) - 4\zeta(3).\]
For $p\geq 3$ the relation between the cumulants of $C_n$ and $\tilde{C}_n$ is
\[ \kappa_p(\tilde{C}_n) = \kappa_p(C_n)\cdot n^{p/2}\]
and therefore in order to establish the desired Gaussian scaling limit \eqref{rescaled}, it would suffice to show that for every $p\geq 3$, as $n\to \infty$, 
\begin{equation} \label{suffice}
\kappa_p(C_n)=o(n^{-p/2}).
\end{equation} 
It is worth mentioning already here that according to a theorem of Svante Janson \cite{Janson88}, it suffices to establish \eqref{suffice} for all sufficiently large $p$, and it will automatically hold for all $p\geq 3$. 
 
Actually we expect stronger bounds than \eqref{suffice} to hold. If as a comparison to $C_n$ we let $S_n$ be a sum of $n$ independent mean $1/n$ exponential variables, then the moment generating function of $S_n$ is 
\[
\frac1{(1-t/n)^n},\] and a simple calculation shows that its cumulants are given by 
\[\kappa_p(S_n) = \frac{(p-1)!}{n^{p-1}}.\]
Roughly speaking we expect the cost $C_n$ of bipartite matching to be qualitatively as regular as $S_n$, and therefore we expect the $p$:th cumulant of $C_n$ to be of order $1/n^{p-1}$ for fixed $p$ in the large $n$ limit. This is known to hold for $p=1$ and $p=2$ with constants $\zeta(2)$ and $4\zeta(2)-4\zeta(3)$ respectively. For $p=3$, a preliminary calculation indicates that 
\[ \kappa_3(C_n) \sim \frac{24\zeta(2) - 48\zeta(3) + 18\zeta(4) + 48\zeta(5) - 24\zeta(2)\zeta(3)}{n^2}.\] 
Here the numerator is approximately $3.5787$ and the prediction agrees well with exact values for small $n$, but the result will have to be confirmed as some supposedly minor terms were discarded without proper estimates.

\section{The recursion and its applications} \label{S:consequences}

In this section we generalize the minimum cost perfect matching problem to incomplete matchings involving a ``ghost vertex''. We explore a certain recursive equation for their distributions, but postpone its proof to Section~\ref{S:proof}.

\subsection{Matchings of $k$ edges on $K_{m,n}$} \label{S:bipartiteMatching}
We consider a complete bipartite graph of $m$ by $n$ vertices and assume that the $mn$ edges are assigned independent mean 1 exponentially distributed costs. For $k\leq \min(m,n)$ we ask for the minimum total cost of a $k$-matching, by which we mean a set of $k$ edges of which no two have a vertex in common. We denote this minimum cost by $C_{k,m,n}$. In particular $C_{n,n,n}$ is the minimum cost of an $n$ by $n$ perfect matching, which was denoted $C_n$ in Section~\ref{S:prelGaussian}. 

The mean and variance of $C_{k,m,n}$ are already quite well understood (see \cite{W10} and its references). 
We show here that the density function of $C_{k,m,n}$ can be found explicitly using a certain recursion. The recursion works in time polynomial in $\max(k,m,n)$, but we do not further discuss questions of computational efficiency.

\subsection{Adding a ghost vertex} \label{S:ghost}
Following an idea from \cite{W08easy} (in a sense going back to the paper \cite{BCR02} by Marshall Buck, Clara Chan and David Robbins), we introduce a special extra vertex $g$ that we call the \emph{ghost vertex}. Denoting the two vertex sets of the original graph by $A = \{a_1,\dots, a_m\}$ and $B = \{b_1, \dots, b_n\}$, we extend the first vertex set to $A+g = \{a_1,\dots, a_m, g\}$, and in the extended graph we connect the ghost vertex to all vertices in $B$. 

The ghost vertex is special in two respects: First, the edges from $g$ have costs that are exponential of mean $1/h$ where eventually $h$ will tend to zero (so that the ``ghost edges'' become expensive). And second, the way we modify the matching problem to the extended graph is that we allow any number of edges from $g$ (but still at most one from every other vertex). We deviate from standard terminology and still call this a ``matching''.

When $h$ is small, the edges from $g$ will most likely be very expensive. The probability that some edge from $g$ participates in the minimum cost $k$-matching turns out to be of order $h$, and more generally the probability that exactly $d$ edges from $g$ participate is of order $h^d$ (assuming that $m\geq k$ so that solutions to the original problem exist). 

For $h>0$ we let $I_{k,m,n}(d)$ be the indicator variable for the event that exactly $d$ edges from $g$ participate in the optimal solution to the $k$-matching problem. We let $C_{k,m,n}$ denote the cost of the solution to the extended matching problem, even though the whole model now depends implicitly on $h$. This is a slight ``abuse of notation'', but setting $h=0$ will make the ghost edges infinitely expensive and give back the original model. 

For $d\geq 0$ we introduce the ``ghost generating function'', which is a function of the complex variable $t$ defined (where it converges) by
\begin{equation} \label{ghostGenerating} 
F_{k,m,n}^{(d)}(t) = \lim_{h\to 0} \frac1{h^d}\cdot \mathbb{E}\left[\exp(t\cdot C_{k,m,n})\cdot I_{k,m,n}(d)\right].
\end{equation}
We have chosen this notation because of certain similarities to a $d$:th derivative with respect to $h$.
The case $d=0$ corresponds to the original problem, which is why we call the extra vertex a ``ghost'' vertex. When $h\to 0$, what happens is to a first approximation that the ghost vertex becomes irrelevant as its edges get expensive, and we get the original problem back. The function $F_{k,m,n}^{(0)}(t)$ is therefore the moment generating function of $C_{k,m,n}$ in the ordinary sense. 

When $d\geq 1$ on the other hand, the function $F_{k,m,n}^{(d)}(t)$ is not in general the moment generating function of any probability distribution. This is because we normalize only by dividing by the appropriate power of $h$, which is not the same thing as \emph{conditioning} on $d$ ghost edges participating in the minimum matching. Therefore there is no reason that $F_{k,m,n}^{(d)}(0)$ should equal 1. It is still a moment generating function, but there is a nontrivial zeroth moment.

\subsection{Recursion} \label{S:recursion}
The ghost generating functions $F_{k,m,n}^{(d)}$ can be computed recursively through the following equation which is established in Section~\ref{S:proofRecursion} (Theorem~\ref{T:matchingRecursion}). We abbreviate by making the dependence on the complex variable $t$ implicit.
\begin{multline} \label{matchingRecursion}
F_{k,m,n}^{(d)}
= \left(1 - \frac{d\cdot t}{(m-k+d)\cdot n}\right)^{-1}
\\ \cdot \left( \frac1{m-k+d}\cdot F_{k-1,m,n-1}^{(d-1)}
+ \frac{d+1}{n} \cdot t \cdot F_{k,m,n}^{(d+1)}
+F_{k-1,m,n-1}^{(d)} \right)
\end{multline}
The base case is that $F_{0,m,n-k}^{(0)} = 1$.
Moreover, certain terms and conditions apply. When $d=k$, the second and third terms of the second factor are treated as zero, and when $d=0$,
\[ F_{k,m,n}^{(0)}
=  \frac{t}{n} \cdot F_{k,m,n}^{(1)}
+F_{k-1,m,n-1}^{(0)} 
\]
 This too might seem to follow from \eqref{matchingRecursion} by treating impossible terms as zero, but to be accurate we cannot apply \eqref{matchingRecursion} when $d=0$ and $k=m$, as that would involve a division by zero.

\subsection{Example: 3 by 3 perfect matching} \label{S:example}
We prove equation \eqref{matchingRecursion} in Sections~\ref{S:nesting}--\ref{S:proofRecursion}, but first make some comments on how the recursion works and what information we get from it. The recursion is essentially ``two-dimensional'': In order to find $F_{k,m,n}^{(0)}$ we compute a triangular array of functions $F_{k-i,m,n-i}^{(d)}$ for $0\leq i \leq k$ and $0\leq d \leq k-i$. Including the base case we need to compute $(k+1)(k+2)/2$ functions, all of them rational in the variable $t$. 

When for instance $k=m=n=3$ as in Figure~\ref{F:matchings}, we find that 
\[ F_{0,3,0}^{(0)} = 1, \qquad F_{1,3,1}^{(1)} = \frac{1}{3-t}, \qquad F_{2,3,2}^{(2)} = \frac{1}{(3-t)^2}, \qquad F_{3,3,3}^{(3)} = \frac{1}{(3-t)^3},\]
\[ F_{1,3,1}^{(0)} = \frac{3}{3-t}, \qquad F_{2,3,2}^{(1)} = \frac{6\,(5-t)}{(4-t)(3-t)^2}, \qquad F_{3,3,3}^{(2)} = \frac{9\,(19-9t+t^2)}{(4-t)(3-t)^4},\]
\[ F_{2,3,2}^{(0)} = \frac{6\,(6-t)}{(4-t)(3-t)^2}, \qquad F_{3,3,3}^{(1)} = \frac{18\,(99-65t+14t^2-t^3)}{(4-t)(3-t)^5},\]
and finally
\begin{equation} \label{finally} 
F_{3,3,3}^{(0)} = \frac{6\,(162-90t+16t^2-t^3)}{(4-t)(3-t)^5}.
\end{equation}
The first row consists of the cases where all edges in the solution are ghost edges, the second row are the cases with exactly one ordinary edge, and so on. The last function $F_{3,3,3}^{(0)}$ is the moment generating function for 3 by 3 perfect matching, denoted $F_3(t)$ in Section~\ref{S:prelGaussian}.

In principle, a probability distribution on the real numbers can always be recovered from its characteristic function, which is essentially the moment generating function restricted to purely imaginary $t$ (where it always converges). Recovering the distribution is particularly easy when the moment generating function is a rational function in $t$, as we can read off the density function from the partial fraction decomposition. In our example we can convert \eqref{finally} to
\[ F_{3,3,3}^{(0)} = \frac{54}{(3-t)^5} + \frac{72}{(3-t)^4} - \frac{30}{(3-t)^3} + \frac{36}{(3-t)^2} - \frac{36}{3-t} + \frac{36}{4-t},\]
from which we find that the density function of $C_{3,3,3}$ for $x\geq 0$ is
\begin{equation} \label{densityExample} 
f_{3,3,3}^{(0)}(x) = \left(\frac{9x^4}{4} + 12x^3 - 15x^2 + 36x -36\right) e^{-3x} + 36e^{-4x}.
\end{equation}
In general, a term \[ \frac1{(a-t)^b} \] in the moment generating function has the ``inverse'' \[ \frac{x^{b-1}}{(b-1)!}\cdot e^{-ax}\] in the density.
For instance, the coefficient $9/4$ in \eqref{densityExample} comes from $54/4!$. 

If we plot the function \eqref{densityExample} as in Figure~\ref{F:density3}, we get a sanity check by noting that the function indeed seems to take nonnegative values for $x\geq 0$. 
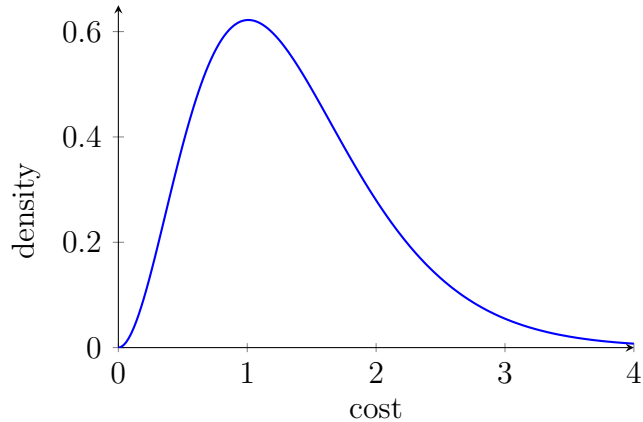
\begin{figure} [h]
\begin{center}
\begin{tikzpicture}
\begin{axis}[
    axis lines = left,
    xscale=1.0,
    yscale=0.8,
    xlabel = {cost}, 
    ylabel = {density},
    ymax=0.65,
]
\addplot [
        color=blue, thick 
    ]
    coordinates
    {
                          (0, 0.)
                       (0.02, 0.0012)
                       (0.04, 0.0046)
                       (0.06, 0.0101)
                       (0.08, 0.0176)
                       (0.10, 0.0269)
                       (0.12, 0.0379)
                       (0.14, 0.0504)
                       (0.16, 0.0642)
                       (0.18, 0.0793)
                       (0.20, 0.0954)
                       (0.22, 0.1125)
                       (0.24, 0.1304)
                       (0.26, 0.1490)
                       (0.28, 0.1681)
                       (0.30, 0.1877)
                       (0.32, 0.2076)
                       (0.34, 0.2277)
                       (0.36, 0.2479)
                       (0.38, 0.2681)
                       (0.40, 0.2883)
                       (0.42, 0.3083)
                       (0.44, 0.3280)
                       (0.46, 0.3474)
                       (0.48, 0.3665)
                       (0.50, 0.3851)
                       (0.52, 0.4032)
                       (0.54, 0.4207)
                       (0.56, 0.4376)
                       (0.58, 0.4539)
                       (0.60, 0.4696)
                       (0.62, 0.4845)
                       (0.64, 0.4987)
                       (0.66, 0.5122)
                       (0.68, 0.5249)
                       (0.70, 0.5368)
                       (0.72, 0.5479)
                       (0.74, 0.5582)
                       (0.76, 0.5677)
                       (0.78, 0.5764)
                       (0.80, 0.5844)
                       (0.82, 0.5915)
                       (0.84, 0.5978)
                       (0.86, 0.6034)
                       (0.88, 0.6082)
                       (0.90, 0.6123)
                       (0.92, 0.6156)
                       (0.94, 0.6183)
                       (0.96, 0.6202)
                       (0.98, 0.6214)
                       (1.00, 0.6220)
                       (1.02, 0.6220)
                       (1.04, 0.6213)
                       (1.06, 0.6201)
                       (1.08, 0.6183)
                       (1.10, 0.6159)
                       (1.12, 0.6131)
                       (1.14, 0.6097)
                       (1.16, 0.6059)
                       (1.18, 0.6016)
                       (1.20, 0.5969)
                       (1.22, 0.5918)
                       (1.24, 0.5863)
                       (1.26, 0.5805)
                       (1.28, 0.5743)
                       (1.30, 0.5678)
                       (1.32, 0.5611)
                       (1.34, 0.5540)
                       (1.36, 0.5468)
                       (1.38, 0.5393)
                       (1.40, 0.5316)
                       (1.42, 0.5237)
                       (1.44, 0.5157)
                       (1.46, 0.5075)
                       (1.48, 0.4992)
                       (1.50, 0.4907)
                       (1.52, 0.4822)
                       (1.54, 0.4736)
                       (1.56, 0.4649)
                       (1.58, 0.4562)
                       (1.60, 0.4474)
                       (1.62, 0.4386)
                       (1.64, 0.4298)
                       (1.66, 0.4210)
                       (1.68, 0.4122)
                       (1.70, 0.4035)
                       (1.72, 0.3947)
                       (1.74, 0.3860)
                       (1.76, 0.3773)
                       (1.78, 0.3688)
                       (1.80, 0.3602)
                       (1.82, 0.3518)
                       (1.84, 0.3434)
                       (1.86, 0.3351)
                       (1.88, 0.3269)
                       (1.90, 0.3188)
                       (1.92, 0.3107)
                       (1.94, 0.3028)
                       (1.96, 0.2950)
                       (1.98, 0.2874)
                       (2.00, 0.2798)
                       (2.02, 0.2723)
                       (2.04, 0.2650)
                       (2.06, 0.2578)
                       (2.08, 0.2507)
                       (2.10, 0.2438)
                       (2.12, 0.2369)
                       (2.14, 0.2303)
                       (2.16, 0.2237)
                       (2.18, 0.2173)
                       (2.20, 0.2110)
                       (2.22, 0.2048)
                       (2.24, 0.1988)
                       (2.26, 0.1928)
                       (2.28, 0.1871)
                       (2.30, 0.1814)
                       (2.32, 0.1759)
                       (2.34, 0.1705)
                       (2.36, 0.1653)
                       (2.38, 0.1601)
                       (2.40, 0.1551)
                       (2.42, 0.1503)
                       (2.44, 0.1455)
                       (2.46, 0.1409)
                       (2.48, 0.1364)
                       (2.50, 0.1320)
                       (2.52, 0.1277)
                       (2.54, 0.1235)
                       (2.56, 0.1195)
                       (2.58, 0.1155)
                       (2.60, 0.1117)
                       (2.62, 0.1080)
                       (2.64, 0.1044)
                       (2.66, 0.1008)
                       (2.68, 0.0974)
                       (2.70, 0.0941)
                       (2.72, 0.0909)
                       (2.74, 0.0878)
                       (2.76, 0.0847)
                       (2.78, 0.0818)
                       (2.80, 0.0790)
                       (2.82, 0.0762)
                       (2.84, 0.0735)
                       (2.86, 0.0709)
                       (2.88, 0.0684)
                       (2.90, 0.0660)
                       (2.92, 0.0636)
                       (2.94, 0.0613)
                       (2.96, 0.0591)
                       (2.98, 0.0570)
                       (3.00, 0.0549)
                       (3.02, 0.0529)
                       (3.04, 0.0510)
                       (3.06, 0.0491)
                       (3.08, 0.0473)
                       (3.10, 0.0456)
                       (3.12, 0.0439)
                       (3.14, 0.0422)
                       (3.16, 0.0407)
                       (3.18, 0.0391)
                       (3.20, 0.0377)
                       (3.22, 0.0363)
                       (3.24, 0.0349)
                       (3.26, 0.0336)
                       (3.28, 0.0323)
                       (3.30, 0.0311)
                       (3.32, 0.0299)
                       (3.34, 0.0287)
                       (3.36, 0.0276)
                       (3.38, 0.0265)
                       (3.40, 0.0255)
                       (3.42, 0.0245)
                       (3.44, 0.0236)
                       (3.46, 0.0227)
                       (3.48, 0.0218)
                       (3.50, 0.0209)
                       (3.52, 0.0201)
                       (3.54, 0.0193)
                       (3.56, 0.0185)
                       (3.58, 0.0178)
                       (3.60, 0.0171)
                       (3.62, 0.0164)
                       (3.64, 0.0158)
                       (3.66, 0.0151)
                       (3.68, 0.0145)
                       (3.70, 0.0139)
                       (3.72, 0.0134)
                       (3.74, 0.0128)
                       (3.76, 0.0123)
                       (3.78, 0.0118)
                       (3.80, 0.0113)
                       (3.82, 0.0109)
                       (3.84, 0.0104)
                       (3.86, 0.0100)
                       (3.88, 0.0096)
                       (3.90, 0.0092)
                       (3.92, 0.0088)
                       (3.94, 0.0085)
                       (3.96, 0.0081)
                       (3.98, 0.0078)
                       (4.00, 0.0075)
}; 
\end{axis}
\end{tikzpicture}

\caption{The density function $(9x^4/4 + 12x^3 - 15x^2 + 36x -36) e^{-3x} + 36e^{-4x}$ of the cost $C_{3,3,3}$ of the 3 by 3 minimum perfect matching.}
\label{F:density3}
\end{center}
\end{figure}
The density function also appears to be log-concave, but I cannot say if this is of any significance, whether it holds more generally, or if there is a simple reason for it. As was mentioned in Section~\ref{S:prelGaussian}, the densities for $n$ by $n$ perfect matching for $n\leq 25$ are shown in Figure~\ref{F:densities}.

We can find the moments and cumulants of $C_{3,3,3}$ from the power series expansions of $F_{3,3,3}^{(0)}$ and its logarithm respectively. The first few terms are
\[ F_{3,3,3}^{(0)} = 1+\frac{49t}{36} + \frac{1529t^2}{1296} + \frac{12811t^3}{15552} + \dots\]
and
\[ \log F_{3,3,3}^{(0)} = \frac{49t}{36} + \frac{73t^2}{288} + \frac{8185t^3}{139968} + \dots\]
from which it follows that $C_{3,3,3}$ has mean $49/36$, variance $73/144$, and third cumulant $8185/23328$. Notice that the mean $49/36$ is equal to $1+1/4+1/9$ in accordance with equation \eqref{parisi}.

We note in passing that when $F^{(0)}_{3,3,3}$ is written as in \eqref{finally}, the factor 6 in the numerator is the number of 3-matchings. We can see that it must be because when $t\to -\infty$, the expectation of $\exp(t\cdot C_{3,3,3})$ is dominated by the cases where there is an extremely cheap matching, and if $\epsilon\to 0$, the six events of a particular matching having cost at most $\epsilon$ become almost disjoint. Therefore in the limit $t\to -\infty$, the asymptotics is
\[ F^{(0)}_{3,3,3}(t) \sim \frac{6}{(1-t)^3} \sim \frac6{(-t)^3}.\]

We can derive a ``formula'' for the higher cumulants of $C_{3,3,3}$ by
factoring the numerator of \eqref{finally} in terms of its three complex zeros. Rescaling so that each factor is 1 at $t=0$, we obtain 
\begin{equation} \label{factorization}
 F_{3,3,3}^{(0)}(t) = \frac{(1-t/\rho_1)(1-t/\rho_2)(1-t/\rho_3)}{(1-t/3)^5(1-t/4)}.
 \end{equation}
Here one of the zeros is real, say $\rho_1$, and the other two are complex conjugates. Numerically $\rho_1\approx 3.509$ and $\rho_{2,3} \approx 6.246\pm 2.677i$.
The logarithm of \eqref{factorization} can be written
\begin{multline} \label{log}
\log F_{3,3,3}^{(0)}(t)
= \log(1-t/\rho_1) + \log(1-t/\rho_2) + \log(1-t/\rho_3)\\ - 5\log(1-t/3)-\log(1-t/4).
\end{multline}
Using the series $\log(1-x) = -x-x^2/2-x^3/3-\dots$ we find that the coefficient of $t^p$ in the Taylor expansion of \eqref{log} is 
\[ \frac1p\cdot\left(\frac{5}{3^p} + \frac1{4^p} - \frac1{\rho_1^p} - \frac1{\rho_2^p} - \frac1{\rho_3^p}\right).\] 
Since this coefficient is $\kappa_p/p!$, the $p$:th cumulant of $C_{3,3,3}$ is 
\[ \kappa_p = (p-1)! \cdot \left( \frac{5}{3^p} + \frac1{4^p} - \frac1{\rho_1^p} - \frac1{\rho_2^p} - \frac1{\rho_3^p}\right).\] 
We know numerically that $t=3$ is the smallest in absolute value of the zeros and poles of $F_{3,3,3}$ (this also follows from Proposition~\ref{P:zeroFree3} in Section~\ref{S:zf3}). This means that the term $5/3^p$ is the dominant one, and for large $p$, \[\kappa_p(C_{3,3,3}) \sim \frac{5(p-1)!}{3^p}.\]

In this example the zeros are all to the right of the abscissa of convergence $\operatorname{Re}(t) = 3$, and therefore only appear through analytic continuation of $F_{3,3,3}(t)$. But since the function is rational in $t$, continuation to a meromorphic function in the complex plane is automatic.

The moment generating functions (or equivalently, Laplace transforms) of $C_{k,m,n}$ for small $k$, $m$ and $n$ were studied by Sven Erick Alm and Gregory Sorkin \cite{AS02}. Although no result like \eqref{finally} on $F_{3,3,3}$ is stated explicitly in their paper, they showed how to compute $F_{k,m,n}(t)$ for $k\leq 4$ and any $m$ and $n$. 

\subsection{Computing the first few moments of $C_{k,m,n}$}
We can compute the mean and variance of $C_{k,m,n}$ from \eqref{matchingRecursion} without computing all the ghost generating functions. It suffices to compute the Taylor series of the functions $F_{k-i,m,n-i}^{(0)}$ up to the coefficient of $t^2$, those of $F_{k-i,m,n-i}^{(1)}$ to the linear term, and the constant term of $F_{k-i,m,n-i}^{(2)}$. 
In general, if we want to find the first $p$ moments of $C_{k,m,n}$, we only need to compute the Taylor series of $F_{k-i,m,n-i}^{(d)}$ up to the coefficient of $t^{p-d}$ for $0\leq d\leq p$. 

As was pointed out by Don Coppersmith and Gregory Sorkin (\cite{CS99}, see also \cite{AS02, LW04, NPS05}), the formula \eqref{parisi} for the expectation of $C_{n,n,n}$ can be generalized to arbitrary $k,m,n$ by
\begin{equation} \label{CoppersmithSorkin}
 \mathbb{E}(C_{k,m,n}) = \sum_{\substack{i,j\geq 0,\\i+j<k}} \frac1{(m-i)(n-j)}
 \end{equation}
Although not immediately obvious, it follows by induction that setting $k=m=n$ in \eqref{CoppersmithSorkin} gives the single sum \eqref{parisi}.
The formula \eqref{CoppersmithSorkin} in turn can be established inductively from \eqref{matchingRecursion}. The statement is that as $t\to 0$,
\begin{equation}
 F_{k,m,n}^{(0)} = 1 +  t\cdot \sum_{\substack{i,j\geq 0,\\i+j<k}} \frac1{(m-i)(n-j)} + O(t^2),
\end{equation}
and 
\begin{equation}
F_{k,m,n}^{(1)} = \frac1{m} + \frac1{m-1}+\dots + \frac1{m-k+1} + O(t).
\end{equation}
This is essentially how the proof of \eqref{parisi} in \cite{W08easy} works. In \cite{W10} there is a similar expression for the variance of $C_{k,m,n}$.

\section{Proof of the recursion} \label{S:proof}

After having explored some of the consequences of the recursion \eqref{matchingRecursion} in Section~\ref{S:consequences}, we devote Section~\ref{S:proof} to its proof. 

In order to conveniently work with ghost edges and limits like \eqref{ghostGenerating}, we introduce the following notation: If $X$ is a random variable in a model that depends on the parameter $h$, we let 
\begin{equation} \label{notation} 
\mathbb{E}^{(d)} [X] = \lim_{h\to 0} \frac{\mathbb{E} [X]}{h^d}
\end{equation}
whenever this limit is well-defined.
The ghost generating function defined by \eqref{ghostGenerating} can now be written as 
\[
F^{(d)}_{k,m,n}(t) = \mathbb{E}^{(d)} [\exp(t\cdot C_{k,m,n})\cdot I_{k,m,n}(d)].
\]

In the rest of Section~\ref{S:proof} we treat $t$ as a fixed complex number without worrying about where in the complex plane the various expected values involving a factor of the form $\exp(t\cdot C)$ are convergent. It would be possible to specify the abscissa of convergence in each case, but it suffices to note that all such expected values converge when $\operatorname{Re}(t) \leq 0$. 

\subsection{Integrating over the cost of a ghost edge} \label{S:integration}

Recall that $I_{k,m,n}(d)$ is the indicator variable for the event that exactly $d$ ghost edges participate in the optimal solution to the $k$-matching problem on the vertex sets $\{a_1,\dots, a_m, g\}$ and $\{b_1,\dots, b_n\}$, and that we refer to the problem as ``matching'' even though we allow several edges from the ghost vertex. The following lemma relates the $k$-matching problem on this graph to the $(k-1)$-matching problem on the induced subgraph with vertex sets $\{a_1,\dots, a_m, g\}$ and $\{b_1,\dots, b_{n-1}\}$, in other words with the vertex $b_n$ deleted from $B$. 
In the following, these two graphs will be denoted $(A+g,B)$ and $(A+g, B-b_n)$ respectively.
\begin{Lemma} \label{L:integration}
\begin{multline} \label{integration}
\mathbb{E}^{(d)}\left[\exp(t\cdot C_{k,m,n}) \cdot I_{k-1,m,n-1}(d)\right]
\\= \mathbb{E}^{(d)}\left[\exp(t\cdot C_{k-1,m,n-1}) \cdot I_{k-1,m,n-1}(d)\right]
\\+ \frac{d+1}{n} \cdot t \cdot \mathbb{E}^{(d+1)}\left[\exp(t\cdot C_{k,m,n})\cdot I_{k,m,n}(d+1)\right].
\end{multline}
\end{Lemma}
Notice first that the right-hand side of equation \eqref{integration} can be written in terms of ghost generating functions as
\[ F^{(d)}_{k-1,m,n-1}(t) +  \frac{d+1}{n} \cdot t \cdot F^{(d+1)}_{k,m,n}(t). \]
The left-hand side on the other hand has a mismatch of indices as it compares $C_{k,m,n}$ to $I_{k-1,m,n-1}(d)$. The lemma shows that it can still be expressed using ghost generating functions.

\begin {proof} [Proof of Lemma~\ref{L:integration}]
For the purpose of this proof we let $x$ be the cost of the edge $e$ from the ghost vertex $g$ to the last vertex $b_n$ of $B$. Still treating $t$ as a fixed complex number, we let
\[ \phi(x) = \exp(t\cdot C_{k,m,n})\cdot I_{k-1,m,n-1}(d) \] 
and think of this as a function of $x$, conditioning on all other edge costs.

Notice that the factor $I_{k-1, m, n-1}(d)$ is independent of $x$ and just says that we are interested only in the cases where the ghost vertex contributes $d$ edges to the minimum $(k-1)$-matching on the subgraph $(A+g, B - b_n)$. So let us suppose that this is the case as otherwise $\phi(x)$ is constant zero, and call this $(k-1)$-matching $M$.

For small $x$ the edge $e$ will participate in the optimal $k$-matching $N$ on $(A+g, B)$, and $C_{k,m,n}$ will increase linearly. As $x$ increases, we reach a point where $e$ becomes too expensive and no longer participates in $N$, after which $C_{k,m,n}$, and thereby also $\phi(x)$, will be constant. 
In principle, $C_{k,m,n}$ as a function of $x$ has the simple behavior shown in Figure~\ref{F:behavior}.
\begin{figure} [h]
\begin{center}
\begin{tikzpicture}

\draw[->] (0,0)--(0,3);
\draw[->] (0,0)--(4,0);

\node at (2,-0.4) {$x$};
\node at (-0.6,1.5) {$C_{k,m,n}$};

\draw[thick, blue] (0,1)--(1.5,2.5)--(4,2.5);

\end{tikzpicture}

\caption{The graph of $C_{k,m,n}$ as a function of $x$.}
\label{F:behavior}
\end{center}
\end{figure}
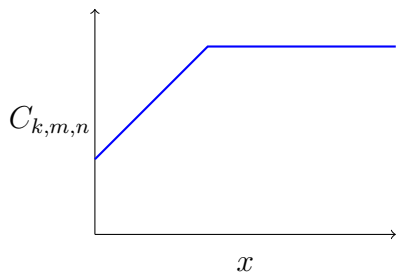

Using the fundamental theorem of calculus we can write
\begin{equation} \label{calculus}
 \phi(\infty) = \phi(0) + \int_0^\infty \phi'(x)\,dx.
 \end{equation}
 
We want to analyze the behavior of each of the three terms of \eqref{calculus} as we take expectation over $x$ assuming exponential distribution of mean $1/h$, and then let $h\to 0$.
 
Starting with the left-hand side, as $h\to 0$, 
\begin{equation} \label{firstTerm}
\mathbb{E}\left[\phi(\infty)\right] \sim \mathbb{E}\left[\exp(t\cdot C_{k,m,n}) \cdot I_{k-1,m,n-1}(d)\right].
\end{equation} 
This is because for small $h$ it is unlikely that the edge $e$ will participate in $N$. Therefore letting $h\to 0$ amounts to essentially the same thing as letting $x\to \infty$.
Notice that even if we condition on $d$ ghost edges participating in an optimal matching on the subgraph $(A+g, B - b_n)$, that does not make it any more likely that the edge $(g, b_n)$ will be of reasonable cost when $h\to 0$.

Next we turn to the right-hand side of \eqref{calculus}. For the first term, \[\phi(0) = \exp(t\cdot C_{k-1,m,n-1}) \cdot I_{k-1,m,n-1}(d),\] since when $e$ has zero cost, we get the minimum $k$-matching on the whole graph $(A+g, B)$ by combining $e$ with the minimum $(k-1)$-matching $M$ on $(A+g, B - b_n)$. No other $k$-matching on $(A+g, B)$ can beat this, since such a matching will have to include as a subset a $(k-1)$-matching on $(A+g, B - b_n)$.  
Taking expected values we get 
\begin{equation} \label{secondTerm}
\mathbb{E}\left[\phi(0)\right] = \mathbb{E}\left[\exp(t\cdot C_{k-1,m,n-1}) \cdot I_{k-1,m,n-1}(d)\right].
\end{equation}
Notice that this too is of order $h^d$ as $h \to 0$.

We now turn to the last term of \eqref{calculus}, still assuming that $I_{k-1,m,n-1}(d) = 1$. A key observation is that the derivative of $C_{k,m,n}$ with respect to $x$ is 1 when the edge $e$ participates in the minimum cost $k$-matching on $(A+g,B)$, and $0$ when it doesn't (again see Figure~\ref{F:behavior}). We are assuming that there are $d$ ghost edges already in $M$, so if $e$ participates in the minimum $k$-matching $N$ on $(A+g, B)$, then $N$ will contain $d+1$ ghost edges.  
Taking expectations over all edge costs except the cost $x$ of the edge $e$, it follows that
\begin{equation} \label{expectation1}
\mathbb{E}\left[\phi'(x)\right] = \frac{d+1}{n}\cdot t\cdot \mathbb{E}\left[\exp(t\cdot C_{k,m,n})\cdot I_{k,m,n}(d+1)\right].
\end{equation}
In this equation the factor $t$ in the right-hand side comes from the inner derivative when differentiating $\exp(t\cdot C_{k,m,n})$ with respect to $C_{k,m,n}$. The factor $(d+1)/n$ occurs because the edge $e$ is a generic ghost edge, and if the matching $N$ contains $d+1$ of the $n$ ghost edges, the probability that $e$ is one of them is $(d+1)/n$.  

Now we want to take expectation in \eqref{expectation1} also over the cost $x$ of the edge $e$, which has a density function of $he^{-hx}$.
Extracting the factor $h$, we find that as $h\to 0$, $e^{-hx}$ will converge to 1 uniformly on every bounded  interval. Therefore
\begin{equation} \label{thirdTerm}
  h \cdot \int_0^\infty \mathbb{E}\left[\phi'(x)\right] dx \sim \frac{d+1}{n}\cdot t\cdot \mathbb{E}\left[\exp(t\cdot C_{k,m,n})\cdot I_{k,m,n}(d+1)\right],
\end{equation}
where in the right-hand side we are now taking expectation also over the cost of the edge $e$.

Combining equations \eqref{firstTerm}, \eqref{secondTerm} and \eqref{thirdTerm}, rescaling by a factor $h^d$ and using the notation \eqref{notation}, we obtain equation \eqref{integration}.
\end{proof}

Lemma~\ref{L:integration} already provides a proof of the case $d=0$ of the recursion \eqref{matchingRecursion}, which states that 
\begin{equation} \label{d=0}
F_{k,m,n}^{(0)}
= \frac{t}{n} \cdot F_{k,m,n}^{(1)}
+F_{k-1,m,n-1}^{(0)}
\end{equation}
This is because in the case $d=0$, the indicator variables $I_{k-1,m,n-1}(0)$ and $I_{k,m,n}(0)$ can both be treated as constant 1 to a first order approximation when $h\to 0$.

\subsection{Nesting and locality} \label{S:nesting}

Lemma~\ref{L:integration} is one of the two main ingredients of the proof of the recursion \eqref{matchingRecursion}, the other being Lemma~\ref{L:independence} of Section~\ref{S:independence}. To prove Lemma~\ref{L:independence} we first we need to establish a couple of combinatorial results on minimum cost matchings.

In the proof, a property called \emph{nesting} is crucial. In its simplest form, involving no ghost vertex and no randomness, it says that in an $m$ by $n$ bipartite graph with fixed edge costs, every vertex that is matched in some minimum $(k-1)$-matching is also matched in some minimum $k$-matching. This is the Nesting Lemma of Marshall Buck, Clara Chan and David Robbins \cite[Lemma~2, Section~3.1]{BCR02}.
Another closely related property is what I have called \emph{locality}, which says that a non-optimal matching can always be improved using at most one new vertex on each side of the graph. 

Nesting and locality both follow from the so-called Hungarian algorithm for bipartite matching, and there are several variations of them in the literature. To make this section self-contained we state and prove Lemmas~\ref{L:locality} and \ref{L:nesting} below, although they would follow from \cite[Lemmas~12.1 and 12.2]{W10}. 
Notice that there is no randomness involved in these lemmas.

\begin{Lemma} [Locality] \label{L:locality}
Suppose that in a complete bipartite graph on vertex sets $A$ and $B$ the edges are assigned nonnegative costs. If $M$ is a $k$-matching which is not of minimum cost, then there is another $k$-matching $M'$ of smaller cost than $M$ such that at most one vertex in $A$ and at most one in $B$ are matched in $M'$ but not in $M$. 
\end{Lemma}

\begin{proof}
Let $M''$ be a minimum cost $k$-matching and consider the symmetric difference $H=M\triangle M''$ of $M$ and $M''$, in other words the set of edges that belong to exactly one of $M$ and $M''$. The components of $H$ are cycles and paths, and a path either has equally many edges from the two matchings, or one more from one than from the other. We can pair up every path with an excess of edges from $M$ with one with an excess of edges from $M''$ so that $H$ is partitioned into ``pseudo-cycles'', meaning edge sets that are balanced in the sense of containing equally many edges from $M$ as from $M''$, and so that distinct pseudo-cycles never share a vertex. If $H'\subseteq H$ is a pseudo-cycle, then $M\triangle H'$ is a $k$-matching that uses at most one vertex from $A$ and one from $B$ that are not matched in $M$. Since in total, the edges of $M\cap H$ cost more than the edges of $M''\cap H$, there must be a pseudo-cycle $H'$ such that $M\triangle H'$ has smaller cost than $M$, and $M\triangle H'$ is the desired matching.
\end{proof}

The Nesting Lemma is proved in essentially the same way, and can also be derived from locality. The version we need here differs slightly from the version in \cite{BCR02} in that we compare matchings where one vertex in $B$ is forbidden in the smaller one.

\begin{Lemma} [Nesting] \label{L:nesting}
Let $A=\{a_1, \dots, a_m\}$ and $B=\{b_1, \dots, b_n\}$ be the vertex sets of a complete bipartite graph with nonnegative edge costs. Let $M$ be a minimum cost $(k-1)$-matching on the induced subgraph $(A,B-b_n)$ on the vertices of $A$ and $B - b_n$. Then there is a minimum cost $k$-matching on the full graph $(A,B)$ that matches all vertices in $A$ that are matched by $M$, and consequently matches exactly one vertex of $A$ that is not matched by $M$.
\end{Lemma}

\begin{proof}
We assume for simplicity that minimum cost matchings of a given number of edges are unique. By continuity this suffices, and in any case it holds with probability~1 in the random models we eventually study.
  
Let $N$ be the minimum cost $k$-matching on $(A,B)$, and let $H$ be the symmetric difference $M\triangle N$.
The components of $H$ are cycles and paths, and again we analyze these components using the idea of \emph{switching}: If $H'$ is a component of $H$, we can form the matchings $M\triangle H'$ and $N\triangle H'$. 

A cycle in $H$ can be switched in any direction, giving a new matching with the same number of edges and matching the same set of vertices. Since we assume that minimum matchings are unique, this contradicts the optimality of $M$ and $N$. Either $M\triangle H'$ will be cheaper than $M$, or $N\triangle H'$ will be cheaper than $N$. The same goes for a path with an even number of edges unless $b_n$ is one of the endpoints.

Consider the paths in $H$ with an odd number of edges and therefore one endpoint in $A$ and one in $B$. There must be one more such path with excess in $N$ than with excess in $M$, since in total $N$ has one more edge than $M$. Suppose there are at least two paths with excess in $N$. Then one of them does not have $b_n$ as an endpoint. In that case there are two paths, that one and one with excess in $M$, that can together be switched in any direction, leading again to a contradiction. Therefore there is exactly one path with an odd number of edges, and it has excess in $N$. 

This path may or may not have $b_n$ as an endpoint, and if it doesn't, there may or may not be another path in $H$ with an even number of edges and having $b_n$ as one of its endpoints (and therefore also the other endpoint in $B$). 
But in all those cases it is clear that the conclusion of the lemma holds: There is exactly one vertex in $A$ that is matched in $N$ but not in $M$.
\end{proof}

These lemmas can be adapted to a situation where the vertex set $A$ is extended with a ghost vertex. If the edge costs are fixed, the ghost vertex is special only in that it is allowed to have several edges in a matching. In order to apply Lemmas~\ref{L:locality} and \ref{L:nesting}, we can replace the single ghost vertex $g$ by $n$ distinct vertices, one for each edge of the original ghost vertex.
We get a bipartite graph on vertex sets $\{a_1,\dots, a_m, g_1,\dots, g_n\}$ and $\{b_1,\dots, b_n\}$. Each vertex $g_i$ has an edge to the vertex $b_i$, and to make the graph complete we may introduce edges of infinite cost from $g_i$ to $b_j$ for $i\neq j$. Applying Lemmas~\ref{L:locality} and \ref{L:nesting} to this setting gives the following conclusions:

\begin{Lemma} [Ghost locality] \label{L:ghostLocality}
Suppose that in a complete bipartite graph on vertex sets $A + g$ and $B$, the edges are assigned nonnegative costs. Allowing multiple edges from $g$, suppose that $M$ is a $k$-matching which is not of minimum cost. Then there is another $k$-matching $M'$ of smaller cost than $M$ such that either there is exactly one vertex in $A$ that is matched in $M'$ but not in $M$, and the ghost edges of the two matchings are the same, or there is exactly one ghost edge in $M'$ which is not in $M$, and the two matchings involve exactly the same vertices of $A$. 
\end{Lemma}

\begin{Lemma} [Ghost nesting] \label{L:ghostNesting}
Consider a complete bipartite graph on vertex sets $A+ g$ and $B$. Allowing multiple edges from $g$, and supposing that the minimum cost matchings are unique, let $M$ be the minimum $(k-1)$-matching on the induced subgraph $(A+g,B-b_n)$, and let $N$ be the minimum $k$-matching on $(A+g,B)$.  

Then every vertex in $A$ that is matched by $M$ is also matched by $N$, and every ghost edge in $M$ is also in $N$. 
\end{Lemma}

In Lemma~\ref{L:ghostNesting} we can thus identify two possibilities: Either there is one more ghost edge in $N$ than in $M$, and in that case the ordinary vertices in $A$ that are matched are the same in the two matchings. Or there is an ordinary vertex in $A$ that is matched by $N$ but not by $M$, and in that case the two matchings contain exactly the same ghost edges.

In both cases there is a well-defined ``new'' edge in $N$: Either the new ghost edge, or the unique edge that matches a vertex in $A$ that is unmatched by $M$ (although there may be several other edges in $N$ that are not in $M$).

\subsection{Independence lemmas} \label{S:independence}

Continuing the proof of equation~\eqref{matchingRecursion}, we apply an argument that stems from \cite{W08easy, W08complete} and was refined in \cite[Theorem~12.3]{W10} and \cite{HW10}. Since we have already established the case $d=0$ of the recursion, we may assume throughout this subsection that $d\geq 1$.

Our goal is to establish Lemma~\ref{L:independence}, but first we prove a more technical ``sub-lemma''. In this lemma we are not taking the limit $h\to 0$, but treat $h$ as a fixed positive number.

\begin{Lemma} \label{L:forward}
Consider a complete bipartite graph with vertex sets $A + g = \{a_1,\dots, a_m, g\}$ and $B = \{b_1,\dots, b_n\}$, and with random edge costs as described earlier for a fixed $h>0$. Let us condition on the number $d$ of ghost edges in the minimum $(k-1)$-matching $M$ on the induced subgraph $(A+g, B - b_n)$, so that consequently this matching has $k-1-d$ ordinary edges and leaves $m-k+d+1$ vertices in $A$ unmatched. Let $N$ be the optimal $k$-matching on $(A+g, B)$.
Then 
\begin{enumerate}
\item The probability that $N$ has $d+1$ ghost edges (the ones in $M$ and another one) is 
\[ \frac{h}{m-k+d+1+h},\]
while for each of the $m-k+d+1$ vertices in $A$ that are unmatched by $M$, the probability of the vertex being matched by $N$ is \[ \frac{1}{m-k+d+1+h}.\]
\item The cost of $N$ is independent of whether or not the new edge in $N$ is a ghost edge.
\end{enumerate}
\end{Lemma} 

\begin{proof}
In $(A+g, B)$, we condition on 
\begin{enumerate}
\item The set of vertices in $A$ (that is, non-ghost vertices) that are matched in the minimum $(k-1)$-matching $M$ on $(A+g, B - b_n)$, and the costs of all their edges (including those to $b_n$).
\item The set of ghost edges in $M$ and their costs.
\item For each vertex in $B$, the minimum cost of those of its edges that are not included under 1 or 2 (including ghost edges).
\end{enumerate}

It may seem unclear at this point exactly what we are conditioning on, because we have to know first where the minimum $(k-1)$-matching $M$ on $(A+g, B - b_n)$ is located. But actually it suffices to condition on the vertex set, edges, and minima in 1, 2, 3 in order to verify that $M$ is indeed optimal. By Lemma~\ref{L:ghostLocality} (ghost locality), if the vertex set in 1 and the ghost edges in 2 did not give a minimum cost matching, then we would be able to find a cheaper matching using only the minimum cost edges in 3.

This means that conditioning on 1, 2, and 3, we can figure out the cost of $N$, as exactly one of its edges will be ``new'' and must therefore be one of the minimum cost edges in 3. At this point we apply one of the standard properties of the exponential distribution: If we condition on the minimum of a set of exponential variables, then the location of the minimum is distributed with probabilities proportional to the rates (inverse of the mean). Moreover, the value of the minimum is independent of its location. 

Finally we notice that the rates are 1 for each ordinary edge, and $h$ for the ghost edge, and this completes the proof.
\end{proof}

Next we turn to the main result of this subsection.

\begin{Lemma} \label{L:independence} 
If $d\geq 1$,
\begin{multline} \label{independenceLemma}
\mathbb{E}^{(d)}\big[\exp(t\cdot C_{k,m,n})\cdot I_{k,m,n}(d)\big] 
\\=\frac1{m-k+d} \cdot \mathbb{E}^{(d-1)}\big[\exp(t\cdot C_{k,m,n})\cdot I_{k-1,m,n-1}(d-1)\big] 
\\ + \mathbb{E}^{(d)}\big[\exp(t\cdot C_{k,m,n})\cdot I_{k-1,m,n-1}(d)\big].
\end{multline}
\end{Lemma}

Here the left-hand side is the ghost generating function $F^{(d)}_{k,m,n}(t)$, and in the right-hand side we have two terms with the sort of mismatch of indices treated in Lemma~\ref{L:integration}. Again notice that the factors $I_{k-1,m,n-1}(d-1)$ and $I_{k-1,m,n-1}(d)$ refer to the induced subgraph where specifically the vertex $b_n$ of $B$ has been deleted.

\begin{proof} [Proof of Lemma~\ref{L:independence}]
We first treat $h$ as a fixed positive number and only later take the limit $h\to 0$. For convenience we assume that the various minimum cost matchings are unique, since this holds with probability 1. 

Let $M$ be the minimum $(k-1)$-matching on $(A+g, B - b_n)$ and let $N$ be the minimum $k$-matching on $(A+g, B)$. 
Notice first that if $N$ contains exactly $d$ ghost edges, then by nesting, $M$ has to contain either $d-1$ or $d$ ghost edges. 
We can write this as 
\begin{multline} \label{divisionOfCases}
\mathbb{E}\big[\exp(t\cdot C_{k,m,n})\cdot I_{k,m,n}(d)\big] 
\\ = \mathbb{E}\big[\exp(t\cdot C_{k,m,n})\cdot I_{k,m,n}(d)\cdot I_{k-1,m,n-1}(d-1)\big] 
\\ + \mathbb{E}\big[\exp(t\cdot C_{k,m,n})\cdot I_{k,m,n}(d)\cdot I_{k-1,m,n-1}(d)\big].
\end{multline}

Consider first the case that $M$ contains $d-1$ ghost edges. Then by Lemma~\ref{L:forward} with $d$ replaced by $d-1$, the probability that the ``new'' vertex in $N$ is the ghost vertex is 
\[ \frac{h}{m-k+d+h}.\]
This is because there are $m-k+d$ ordinary vertices left in the competition, and we are looking for the probability that the ghost vertex wins.

If on the other hand $M$ contains $d$ ghost edges, there will be $m-k+d+1$ ordinary vertices in the competition, and the probability that one of them wins (and not the ghost vertex) is 
\[ \frac{m-k+d+1}{m-k+d+1+h}.\]

By Lemma~\ref{L:forward}, in both cases the cost of $N$ is unchanged in distribution by conditioning on the location of the ``new'' edge, and therefore
\begin{multline} \label{hd}
\mathbb{E}\big[\exp(t\cdot C_{k,m,n})\cdot I_{k,m,n}(d)\big] 
\\=\frac{h}{m-k+d+h}\,\mathbb{E}\big[\exp(t\cdot C_{k,m,n})\cdot I_{k-1,m,n-1}(d-1)\big] 
\\ + \frac{m-k+d+1}{m-k+d+1+h}\,\mathbb{E}\big[\exp(t\cdot C_{k,m,n})\cdot I_{k-1,m,n-1}(d)\big].
\end{multline}
Here all terms are of order $h^d$ for small $h$. We obtain \eqref{independenceLemma} by dividing by $h^d$ and taking the limit $h\to 0$.
\end{proof}
 
\subsection{Concluding the proof of the recursion} \label{S:proofRecursion}
To establish the recursion \eqref{matchingRecursion} for $d\geq 1$, we combine Lemmas~\ref{L:integration} and \ref{L:independence}. We rewrite both terms of the right-hand side of equation \eqref{independenceLemma} of Lemma~\ref{L:independence} using equation \eqref{integration} of Lemma~\ref{L:integration}. For the first term we have to replace $d$ by $d-1$ in the lemma. This gives

\begin{multline} 
\mathbb{E}^{(d)}\big[\exp(t\cdot C_{k,m,n})\cdot I_{k,m,n}(d)\big] 
\\=\frac1{m-k+d} \, \Big( \mathbb{E}^{(d-1)}\big[\exp(t\cdot C_{k-1,m,n-1}) \cdot I_{k-1,m,n-1}(d-1)\big]
\\+ \frac{d}{n} \cdot t \cdot \mathbb{E}^{(d)}\big[\exp(t\cdot C_{k,m,n})\cdot I_{k,m,n}(d)\big]\Big)
\\ + \mathbb{E}^{(d)}\big[\exp(t\cdot C_{k-1,m,n-1}) \cdot I_{k-1,m,n-1}(d)\big]
\\+ \frac{d+1}{n} \cdot t \cdot \mathbb{E}^{(d+1)}\big[\exp(t\cdot C_{k,m,n})\cdot I_{k,m,n}(d+1)\big].
\end{multline}
This equation has a term $\mathbb{E}^{(d)}\big[\exp(t\cdot C_{k,m,n})\cdot I_{k,m,n}(d)\big]$ also in the right-hand side, and collecting these terms to the left-hand side yields
\begin{multline} 
\left(1-\frac{d\cdot t}{(m-k+d)\cdot n} \right) \cdot \mathbb{E}^{(d)}\big[\exp(t\cdot C_{k,m,n})\cdot I_{k,m,n}(d)\big] 
\\=\frac1{m-k+d} \cdot \mathbb{E}^{(d-1)}\big[\exp(t\cdot C_{k-1,m,n-1}) \cdot I_{k-1,m,n-1}(d-1)\big]
\\ + \mathbb{E}^{(d)}\big[\exp(t\cdot C_{k-1,m,n-1}) \cdot I_{k-1,m,n-1}(d)\big]
\\+ \frac{d+1}{n} \cdot t \cdot \mathbb{E}^{(d+1)}\big[\exp(t\cdot C_{k,m,n})\cdot I_{k,m,n}(d+1)\big].
\end{multline}

In each term of this equation, the cost $C$ and the indicator variable $I$ have the same indices, which means they are all ghost generating functions as defined by \eqref{ghostGenerating}. Dividing by the first factor of the left-hand side and interpreting the terms as ghost generating functions gives equation \eqref{matchingRecursion}. Technically we have assumed in this argument that $d\geq 1$, and we restate the recursion as a theorem divided into the two cases $d=0$ and $d\geq 1$.

\begin{Thm} \label{T:matchingRecursion}
For bipartite matching, the ghost generating functions can be computed from the following equations:

\begin{equation} \label{recursion0}
F_{k,m,n}^{(0)}
= \frac{t}{n} \cdot F_{k,m,n}^{(1)}
+F_{k-1,m,n-1}^{(0)},
\end{equation}
and for $d\geq 1$,
\begin{multline} \label{recursionThm}
F_{k,m,n}^{(d)}
= \left(1 - \frac{d\cdot t}{(m-k+d)\cdot n}\right)^{-1}
\\ \cdot \left( \frac1{m-k+d}\cdot F_{k-1,m,n-1}^{(d-1)}
+ \frac{d+1}{n} \cdot t \cdot F_{k,m,n}^{(d+1)}
+F_{k-1,m,n-1}^{(d)} \right)
\end{multline}
\end{Thm} 

\section{Observations and conjectures} \label{S:observations}
Having established the recursion for the ghost generating functions, we explore some of its consequences.
\subsection{The structure of $F_{k,m,n}(t)$}
From Theorem~\ref{T:matchingRecursion} it follows by induction that $F^{(d)}_{k,m,n}(t)$ is always a rational function in $t$, but we can be more precise.

\begin{Prop} \label{P:polynomialDegrees}
The moment generating function $F_{k,m,n}(t)$ can be written
 \[ F_{k,m,n}(t) = \frac{P(t)}{Q(t)},\]
where $P(t)$ and $Q(t)$ are polynomials of degree $\binom{k}{2}$ and $\binom{k+1}{2}$ respectively, and
\begin{equation} \label{defaultDenominator} 
Q(t) = \prod_{\substack{i,j\geq 0\\ i+j<k}} \left(1-\frac{(k-i-j)\cdot t}{(m-i)(n-j)}\right)
\end{equation}
\end{Prop}

\begin{proof}
The factors of $Q(t)$ are picked up one by one in the recursion. Notice that if we compute $F_{k,m,n}$ as suggested by \eqref{recursionThm}, the value of $n$ is changing throughout, while $m$ stays the same.

The fact that $\deg(Q) - \deg(P) = k$ can be proved by induction too, but also follows from a direct argument about the cost of the minimum matching.
When $t$ is large negative, the behavior of $F_{k,m,n}(t) = \mathbb{E}\left(\exp(t\cdot C_{k,m,n})\right)$ is governed by the cases where the cost of the minimum $k$-matching is unusually small. Here we note that $C_{k,m,n}$ is bounded from above by the cost of any given $k$-matching, and from below by the cost of the $k$:th cheapest of the $mn$ edges, which in turn stochastically dominates a sum of $k$ independent exponential variables of mean $1/(mn)$. Since for negative $t$ the inequalities for $F_{k,m,n}(t)$ go the opposite way to those for $C_{k,m,n}$,
\begin{equation} 
\left(1-t\right)^{-k} \leq F_{k,m,n}(t) \leq \left(1-\frac{t}{mn}\right)^{-k}
\end{equation}
showing that $F_{k,m,n}(t)$ is of order $\left|t\right|^{-k}$ as $t\to -\infty$.
\end{proof}

We may ask whether these polynomial degrees are the smallest possible, or if factors of $Q(t)$ sometimes cancel against equal factors of $P(t)$. In the recursion, $m$ and $n$ do not have to be integers, and if we take $k=m=3$ and $n=10/3$, there is actually a cancellation giving 
\[ F_{3,3,10/3} = \frac{280\,(350-85t+6t^2)}{(7-2t)(14-3t)(10-3t)^3}.\]
We would have expected also a factor $(4-t)$ in the denominator, but written in that way, $t=4$ happens to be a zero of the numerator.
I see no reason that such a cancellation could not occur also when $m$ and $n$ are both integers. 

\subsection{Cumulants, zeros, and scaling limits} \label{S:cumulants}
Using the recursion and the computer algebra system Mathematica, we have computed the moment generating functions $F_{k,m,n}(t) = F_{k,m,n}^{(0)}(t)$ for all $k\leq m \leq n \leq 15$ and less systematically for some higher values. We state some of our observations as conjectures. Recall that the cumulants $\kappa_p$ of the random variable $C_{k,m,n}$ are defined by the power series
\[ \log F_{k,m,n}(t) = \sum_{p=1}^\infty \frac{\kappa_p t^p}{p!}.\] 

\begin{Conjecture} \label{C:positive}
All the cumulants of $C_{k,m,n}$ are positive.
\end{Conjecture}

It also seems that for fixed $k, m, n$, the sequence $\kappa_p/p!$ of coefficients of the cumulant generating function is log-concave. As another addendum to Conjecture~\ref{C:positive}, numerical evidence suggests that in the case $k=m=n$ of perfect matching on $K_{n,n}$, the cumulants satisfy
\[ \frac{(p-1)!}{n^{p-1}} \leq \kappa_p(C_{n,n,n}) \leq \frac{2(p-1)!}{n^{p-1}}. \]

We have computed numerically the zeros of the moment generating function $F_{k,m,n}$ in a large number of instances. Figure~\ref{F:20by20} shows the zeros of the function $F_{20,20,20}(t)$. We propose the following conjecture.
\begin{Conjecture} \label{C:zeroFree}
The moment generating function $F_{k,m,n}(t)$ has no complex zeros in the open disk of radius $mn/k$ centered at the origin.
\end{Conjecture}

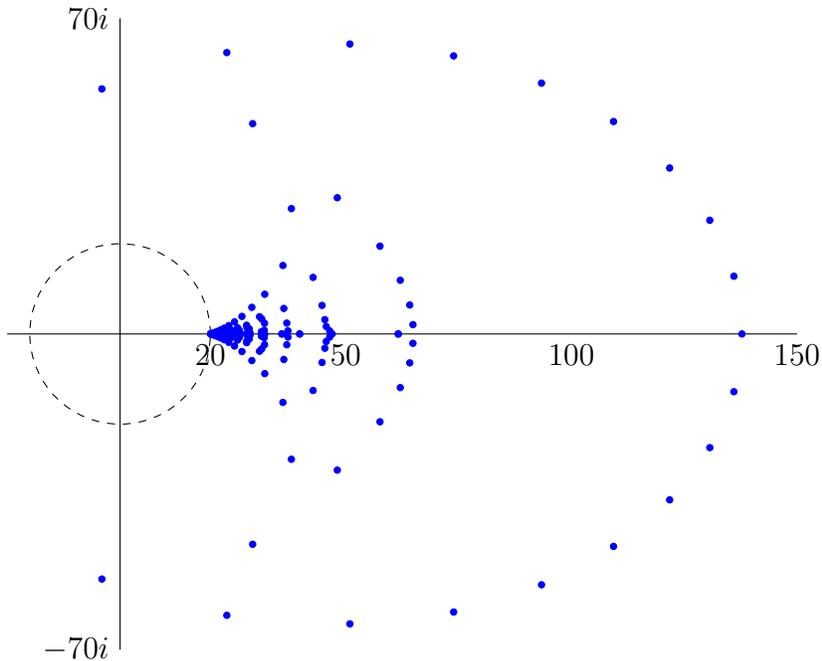
\begin{figure} [h]
\begin{center}
\begin{tikzpicture} [scale=0.06]

\draw (-25,0)--(150,0);
\draw  (0,-70)--(0,70);

\node[below] at (20,0) {20};
\node[below] at (50,0) {50};
\node[below] at (100,0) {100};
\node[below] at (150,0) {150};

\node[left] at (0,70) {$70i$};
\node[left] at (0,-70) {$-70i$};

\draw[black, dashed] (0,0) circle (20);

\foreach \x/\y in 
{20.0408/0, 20.1609/0, 20.1843/0, 20.233/0, 20.5266/0, 21.1493/0, 21.2259/0, 
21.2275/0, 21.9206/0, 21.9449/0, 22.0008/0, 22.4552/0, 22.4559/0, 23.337/0, 23.337/0, 
23.4942/0, 23.5564/0, 23.6576/0, 26.2265/0, 26.4224/0, 26.5414/0, 27.9947/0, 
27.9949/0, 32.0627/0, 35.845/0, 35.848/0, 46.9986/0, 137.797/0, - 4.02611 / -  
  54.343 , -4.02611/ + 54.343, 20.0973/ - 0.0184486 , 
 20.0973/ + 0.0184486 , 20.2163/ - 0.0395084 , 20.2163/ + 0.0395084 , 
 20.3188/ - 0.014208 , 20.3188/ + 0.014208 , 20.3602/ - 0.0928047 , 
 20.3602/ + 0.0928047 , 20.4153/ - 0.00158372 , 
 20.4153/ + 0.00158372 , 20.5037/ - 0.0302406 , 20.5037/ + 0.0302406 ,
  20.5748/ - 0.154321 , 20.5748/ + 0.154321 , 20.6646/ - 0.0626236 , 
 20.6646/ + 0.0626236 , 20.6743/ - 0.026586 , 20.6743/ + 0.026586 , 
 20.8376/ - 0.0162449 , 20.8376/ + 0.0162449 , 20.8624/ - 0.0334553 , 
 20.8624 /+ 0.0334553 , 20.8655 /- 0.264154 , 20.8655/ + 0.264154 , 
 21.0085/ - 0.000419944 , 21.0085/ + 0.000419944 , 
 21.1047 / -  0.103634 , 21.1047 / + 0.103634 , 21.2318 / -  0.404405 , 
 21.2318 / + 0.404405 , 21.4293 / -  0.193746 , 21.4293 / + 0.193746 , 
 21.4785 / -  0.0841032 , 21.4785 / + 0.0841032 , 21.5199 / -  0.0111744, 
 21.5199 / + 0.0111744 , 21.7153 / -  0.607446 , 21.7153 / + 0.607446 , 
 21.8424 / -  0.209768 , 21.8424 / + 0.209768 , 21.8569 / -  0.155137 , 
 21.8569 / + 0.155137 , 22.3197 / -  0.894315 , 22.3197 / + 0.894315 , 
 22.3227 / -  0.16867 , 22.3227 / + 0.16867 , 22.3274 / -  0.223144 , 
 22.3274 / + 0.223144 , 22.8016 / -  0.34551 , 22.8016 / + 0.34551 , 
 22.8072 / -  0.116963 , 22.8072 / + 0.116963 , 22.9228 / -  0.190631 , 
 22.9228 / + 0.190631 , 23.092 / -  1.28826 , 23.092 / + 1.28826 , 
 23.5477 / -  0.600217 , 23.5477 / + 0.600217 , 23.6545 / -  62.3952 , 
 23.6545 / + 62.3952 , 23.6655 / -  0.313812 , 23.6655 / + 0.313812 , 
 24.0839 / -  1.8539 , 24.0839 / + 1.8539 , 24.5502 / -  0.895619 , 
 24.5502 / + 0.895619 , 24.5916 / -  0.233048 , 24.5916 / + 0.233048 , 
 24.6095 / -  0.305098 , 24.6095 / + 0.305098 , 24.6204 / -  0.630496 , 
 24.6204 / + 0.630496 , 24.8248 / -  0.203669 , 24.8248 / + 0.203669 , 
 25.362 / -  2.67902 , 25.362 / + 2.67902 , 26.0265 / -  1.35965 , 
 26.0265 / + 1.35965 , 26.162 / -  1.05732 , 26.162 / + 1.05732 , 
 26.2887 / -  0.145191 , 26.2887 / + 0.145191 , 26.4838 / -  0.679728 , 
 26.4838 / + 0.679728 , 27.0072 / -  3.89188 , 27.0072 / + 3.89188 , 
 28.0424 / -  1.9371 , 28.0424 / + 1.9371 , 28.1206 / -  1.41898 , 
 28.1206 / + 1.41898 , 28.6381 / -  1.10637 , 28.6381 / + 1.10637 , 
 28.774 / -  0.310674 , 28.774 / + 0.310674 , 29.1994 / -  5.90088 , 
 29.1994 / + 5.90088 , 29.385 / -  46.6372 , 29.385 / + 46.6372 , 
 30.9035 / -  3.82358 , 30.9035 / + 3.82358 , 31.3027 / -  0.440049 , 
 31.3027 / + 0.440049 , 31.3651 / -  0.473673 , 31.3651 / + 0.473673 , 
 31.3936 / -  3.32316 , 31.3936 / + 3.32316 , 31.9446 / -  0.821631 , 
 31.9446 / + 0.821631 , 32.017 / -  2.36837 , 32.017 / + 2.36837 , 
 32.0599 / -  8.79992 , 32.0599 / + 8.79992 , 36.0794 / -  15.1823 , 
 36.0794 / + 15.1823 , 36.2827 / -  5.65789 , 36.2827 / + 5.65789 , 
 36.9443 / -  2.40266 , 36.9443 / + 2.40266 , 37.1972 / -  0.686119 , 
 37.1972 / + 0.686119 , 37.9655 / -  27.7954 , 37.9655 / + 27.7954 , 
 39.7989 / -  0.00113051 , 39.7989 / + 0.00113051 , 42.7723 / -  12.5372 , 
 42.7723 / + 12.5372 , 44.755 / -  6.32544 , 44.755 / + 6.32544 , 
 45.3818 / -  3.17397 , 45.3818 / + 3.17397 , 45.6913 / -  1.66257 , 
 45.6913 / + 1.66257 , 46.452 / -  0.675812 , 46.452 / + 0.675812 , 
 48.1381 / -  30.1999 , 48.1381 / + 30.1999 , 50.9289 / -  64.2933 , 
 50.9289 / + 64.2933 , 57.5918 / -  19.4675 , 57.5918 / + 19.4675 , 
 61.6339 / -  0.0356533 , 61.6339 / + 0.0356533 , 62.0905 / -  11.9 , 
 62.0905 / + 11.9 , 64.1937 / -  6.41734 , 64.1937 / + 6.41734 , 
 64.8986 / -  2.07012 , 64.8986 / + 2.07012 , 73.913 / -  61.66 , 
 73.913 / + 61.66 , 93.3853 / -  55.6186 , 93.3853 / + 55.6186 , 
 109.345 / -  47.1087 , 109.345 / + 47.1087 , 121.78 / -  36.7939 , 
 121.78 / + 36.7939 , 130.674 / -  25.2058 , 130.674 / + 25.2058 , 
 136.015 / -  12.8022 , 136.015 / + 12.8022 }
\filldraw[blue] (\x,\y) circle (0.7);

\end{tikzpicture}

\caption{The $190 = \binom{20}{2}$ complex zeros of the moment generating function $F_{20,20,20}$ of the minimum cost of a 20 by 20 perfect matching. 
There are 28 real zeros of which the smallest is at approximately $20.04$ and the largest at $137.80$, and 81 complex conjugate pairs. 
The first pole of $F_{20,20,20}$ is at $t=20$ and has order 39. There are several zeros close to this pole but none in the disk of radius 20 around the origin. Even so, there is a pair of ``overhang'' zeros whose real part is negative, and in particular well below 20.
}
\label{F:20by20}
\end{center}
\end{figure}

The point $t=mn/k$ is the location of the first pole of $F_{k,m,n}(t)$.
A statement equivalent to Conjecture~\ref{C:zeroFree} is therefore that the Taylor series for $F_{k,m,n}(t)$ and $\log F_{k,m,n}(t)$ have the same radius of convergence.

A main conjecture motivating this work is that after rescaling, $C_{n,n,n}$ is asymptotically Gaussian. This was discussed in Section~\ref{S:prelGaussian}, and we restate the conjecture \eqref{rescaled}.
\begin{Conjecture} \label{C:clt}
As $n\to\infty$,
\begin{equation} \label{eq:clt}
 \sqrt{n}\cdot \left(C_{n,n,n} - \pi^2/6\right) \overset{\mathcal{D}}\longrightarrow \mathcal{N}(0,4\zeta(2)-4\zeta(3)).
 \end{equation}
\end{Conjecture}
As was mentioned in Section~\ref{S:prelGaussian}, the constant $4\zeta(2)-4\zeta(3)\approx 1.7715$ comes from the known asymptotics of the variance of $C_{n,n,n}$ \cite{W10}. 

Conjectures~\ref{C:positive}, \ref{C:zeroFree} and \ref{C:clt} may not seem directly related to one another, but next we show that Conjecture~\ref{C:positive} implies Conjecture~\ref{C:zeroFree}, which in turn implies Conjecture~\ref{C:clt}.

\begin{Prop}
Let $X$ be a real valued random variable whose moment generating function $F$ is a rational function. If all the cumulants of $X$ are nonnegative, then $F$ has no zero of modulus smaller than the pole closest to the origin.
\end{Prop}

\begin{proof}
Since $F$ is a moment generating function, $F(0)=1$. By the fundamental theorem of algebra, $F$ can be factored as
\[ F(t) = \frac{\prod_{F(\rho)=0} (1-t/\rho)}{\prod_{F(\xi)=\infty} (1-t/\xi)}\]
assuming that multiplicities are taken into account.
There must be some neighborhood of the origin where we can take logarithms, so that
\[ \log F(t) = \sum_{F(\rho)=0} \log(1-t/\rho) - \sum_{F(\xi)=\infty} \log(1-t/\xi).\]
From the Taylor expansion $\log(1-x) = -x-x^2/2-x^3/3-\dots$, we find that the coefficient of $t^p$ in the Taylor series of $\log F(t)$ is 
\[ \frac1p \cdot \left(\sum_{F(\xi) = \infty}\frac1{\xi^p} - \sum_{F(\rho)=0} \frac1{\rho^p}\right),\] 
from which it follows that
\begin{equation} \label{zerosAndPoles}
 \frac{\kappa_p}{(p-1)!} = \sum_{F(\xi) = \infty}\frac1{\xi^p} - \sum_{F(\rho)=0} \frac1{\rho^p}.
 \end{equation} 

If there is a zero of $F$ of smaller absolute value than all the poles, then for large $p$ the terms corresponding to the smallest zeros will dominate the right-hand side of \eqref{zerosAndPoles}.
There might conceivably be many (necessarily non-real) zeros of minimum absolute value, but the Dirichlet approximation theorem shows that there must be infinitely many integer values of $p$ for which the $p$:th powers of all such zeros have their phase arguments in, say, the range $-2\pi/8$ to $2\pi/8$. For large $p$ this will make the cumulant $\kappa_p$ negative.
\end{proof}

\begin{Prop}
Suppose that the moment generating functions $F_{n,n,n}(t)$ never have zeros of modulus smaller than $n$, the first pole. Then as $n\to\infty$, $C_{n,n,n}$ has a Gaussian scaling limit as given by equation \eqref{eq:clt}.
\end{Prop}

\begin{proof}
We know from Proposition~\ref{P:polynomialDegrees} that if we apply equation \eqref{zerosAndPoles} to the moment generating function $F_{n,n,n}$ of $C_{n,n,n}$, the right-hand side will have $n^2$ terms in total. The pole of $F$ which is closest to the origin is $\xi_1 = n$. Assuming that no zero of $F$ is closer to the origin than this, the cumulants are bounded by
\[ \left|\kappa_p(C_n)\right| \leq (p-1)!  \cdot \frac{n^2}{n^p} = \frac{(p-1)! }{n^{p-2}}.\]
As was explained in Section~\ref{S:prelGaussian}, this would imply that the cumulant of the rescaled cost $\tilde{C}_n = \sqrt{n}\cdot (C_n - \pi^2/6)$ satisfies 
\[
\kappa_p(\tilde{C}_n) = O\left(\frac1{n^{p/2-2}}\right).
\] 
This tends to zero for $p\geq 5$. According to a theorem of Svante Janson (the main theorem of \cite{Janson88}), this would imply that $\kappa_3(\tilde{C}_n)$ and $\kappa_4(\tilde{C}_n)$ too tend to zero, and that $\tilde{C}_n$ has a Gaussian limit.
\end{proof}

\subsection{A zero-free disk for $k\leq 3$} \label{S:zf3}
Conjecture~\ref{C:zeroFree} on a zero-free disk of $F_{k,m,n}(t)$ is trivial for $k=1$ since the moment generating function $F_{1,m,n}(t)$ has no zero. It clearly holds also for $k=2$. There is only one zero of $F_{2,m,n}$ and it must therefore be real and to the right of the abscissa of convergence. A simple calculation reveals that the first pole is at $t=mn/2$ and the zero is at $t=mn$. 
The first interesting case of Conjecture~\ref{C:zeroFree} is therefore when $k=3$, and in this section we give a proof that it holds for $k=3$ and every $m$ and $n$. 

In \cite{AS02}, Sven Erick Alm and Gregory Sorkin describe a function, suggested to them by Svante Janson, that has a number of properties in common with $F_{k,m,n}$. They discussed the Laplace transform of $C_{k,m,n}$ which is equivalent to the moment generating function (and to the characteristic function) under a simple change of variable.

For given $k$, $m$ and $n$, Janson's function is a rational function that has the same denominator as $F_{k,m,n}$, and numerator equal to the denominator of $F_{k-1,m,n}$. Specifically, let
\[ J_{k,m,n}(t) = \frac{\prod_{\substack{i,j \geq 0\\ i+j< k-1}} \left(1-\frac{(k-1-i-j)\cdot t}{(m-i)(n-j)}\right)}
{\prod_{\substack{i,j \geq 0\\ i+j< k}} \left(1-\frac{(k-i-j)\cdot t}{(m-i)(n-j)}\right)}
 \]
 As Janson, Alm and Sorkin noted (and is easily verified using Theorem~\ref{T:matchingRecursion}),
 \[ J_{1,m,n} = \frac{1}{1-\frac{t}{mn}} = F_{1,m,n} \]
 and
 \[ J_{2,m,n} = \frac{1-\frac{t}{mn}}{\left(1-\frac{2t}{mn}\right)\left(1-\frac{t}{(m-1)n}\right)\left(1-\frac{t}{m(n-1)}\right)} = F_{2,m,n} \]
The initial idea seems to have been that perhaps $F_{k,m,n} = J_{k,m,n}$ for all $k,m,n$. This would have implied a Gaussian scaling limit since all the zeros of $J_{k,m,n}$ are real. But as was noted already in \cite{AS02}, the identity fails for $k=3$. As we have seen, $F_{3,3,3}$ has a pair of conjugate non-real zeros.

Even so, $F_{3, m, n}$ and $J_{3,m,n}$ are close enough that a comparison reveals that $F_{3,m,n}$ cannot have zeros closer to the origin than its first pole.
\begin{Prop} \label{P:zeroFree3}
The moment generating function $F_{3,m,n}(t)$ has no zero in the disk \[\left|t\right| \leq \frac{mn}{3}.\]
\end{Prop}

\begin{proof}
A direct computation shows that $F_{3, m, n}(t)$ can be written as
\[ F_{3,m,n}(t) = \frac{P(t)}{Q(t)},\] where the denominator is 
\begin{multline} 
Q(t) = (mn-3t)\,((m-1)n - 2t) \, (m(n-1)-2t) \\ 
((m-2)n-t) \, ((m-1)(n-1) - t) \, (m(n-2) - t),
\end{multline}
and the numerator is 
\begin{multline} P(t) = m^2n^2(m-1)(n-1)
-mn(4mn-3m-3n+2)t \\
+ (5mn-2m-2n-1)t^2 + 2t^3.
\end{multline}
Written with the same denominator $Q(t)$, the numerator of the Janson function $J_{3, m, n}(t)$ is
\begin{multline}
P_J(t) = (mn-2t)\cdot ((m-1)n-t)\cdot (m(n-1)-t) \\
= m^2n^2(m-1)(n-1)-mn(4mn-3m-3n+2)t \\ + (5mn-2m-2n)t^2 + 2t^3.
\end{multline} 
The two polynomials $P(t)$ and $P_J(t)$ are strikingly similar and differ by exactly $t^2$. Therefore on a circle $\left|t\right| = R$ the difference between them is
\[ \left| P(t) - P_J(t)\right| = R^2.\]
Moreover, provided that $R<mn/2$, it follows from the factorization of $P_J(t)$ that on the circle of radius $R$,
\[\left|P_J(t)\right| \geq P_J(R),\] 
since each of the three factors is minimized in absolute value when $t=R$.

If we take $R=mn/3$ (which is the radius relevant to Conjecture~\ref{C:zeroFree}), we find that $P_J(R)>R^2$, since
\begin{multline}  
P_J\left(\frac{mn}3\right)
= \left(mn-\frac{2mn}3\right)\cdot \left((m-1)n-\frac{mn}3\right) \cdot \left(m(n-1)-\frac{mn}3\right) \\
= \frac{m^2n^2}3\cdot\left(\frac{2m}{3}-1\right) \cdot \left(\frac{2n}3-1\right) \geq \frac{m^2n^2}3 = 3R^2.
\end{multline} 
Notice for the last inequality that in order for the optimization problem to make sense, both $m$ and $n$ have to be at least 3.

It follows from these estimates that on the circle $\left|t\right| = mn/3$, 
\[ \left| P(t) - P_J(t)\right| <  \left| P_J(t)\right|.\] 
Since the three zeros of $P_J(t)$ are $t=mn/2$, $t=(m-1)n$ and $t=m(n-1)$, all of them outside the disk $\left|t\right| \leq mn/3$, it follows from basic complex analysis that the zeros of $P(t)$ are also outside this disk. As $t$ goes around the boundary of the disk, both $P(t)$ and $P_J(t)$ will wind zero times around the origin.  
\end{proof}

This establishes Conjecture~\ref{C:zeroFree} for $k\leq 3$ and all $m$ and $n$, but in view of the distribution of zeros of $F_{n,n,n}(t)$ as shown in Figure~\ref{F:20by20}, there seems to be no hope of generalizing this argument to higher values of $k$. The zeros of $F_{k,m,n}$ are in general not close to the real line and we should not expect $F_{k,m,n}$ to be well approximated by a function with only real zeros.
 
\subsection{Asymptotics for fixed $k$ and large $m$ and $n$}

It turns out that for fixed $k$, if at least one of $m$ and $n$ is large, the zeros of $F_{k,m,n}(t)$ will cluster near the real line, well to the right of the main pole. We can exploit this to prove that for each $k$ there are only finitely many counterexamples to Conjecture~\ref{C:zeroFree}.
\begin{Prop} \label{P:finitely}
For each $k$ there are only finitely many pairs of integers $m, n\geq k$ for which $F_{k,m,n}(t)$ has zeros in the disk $\left| t \right| \leq mn/k$.
\end{Prop}
Notice that we can say ``zeros'' in plural, since zeros to the left of the main pole are non-real and thereby come in conjugate pairs.
We devote the rest of this section to proving Proposition~\ref{P:finitely}.

In order to compare the zeros of $F_{k,m,n}(t)$ to the main pole at $t=mn/k$, we will make the variable substitution 
\[ t = \frac{mn}k\cdot s. \]
This way the main pole is at $s=1$, and ideally we want to prove that $F_{k,m,n}$ has no zeros in the disk given by 
$\left| s \right| < 1$.

Since this rescaling depends on $k$, $m$ and $n$, we want to describe the recursion in Theorem~\ref{T:matchingRecursion} in a way that keeps these numbers fixed. For given $k$, $m$ and $n$ we therefore let 
\begin{equation}
L_{i,j} = F_{i,m,n-k+i}^{(i-j)}(t).
\end{equation}
In terms of ordinary edges and ghost edges in the optimization problem, $i$ is the total number of edges, and $j$ is the number of ``ordinary'' (non ghost) edges. 

This numbering is quite natural since it means that $L_{i,j}$ is computed from the three functions $L_{i-1,j-1}$, $L_{i,j-1}$ and $L_{i-1,j}$. 
The recursion becomes
\begin{multline} \label{recursionL(t)}
 L_{i,j} = \left(1-\frac{(i-j)t}{(m-j)(n-k+i)}\right)^{-1} \\
 \cdot \left( \frac{L_{i-1,j}}{m-j} + \frac{(i-j+1)t}{n-k+i}L_{i,j-1} + L_{i-1,j-1}\right).
 \end{multline}

The base case is $L_{0,0} = 1$. We then compute $L_{i,j}$ for $0\leq j \leq i \leq k$, and the final result is 
\[ L_{k,k} = F_{k,m,n}(t).\]

We can write each $L_{i,j}$ with a denominator consisting of the factors of the form 
\[ 1-\frac{(i-j)t}{(n-k+i)(m-j)} \]
picked up through the recursion. Naturally we can describe this with new indices $i'$ and $j'$ so that 
\[ L_{i,j} = L_{i,j}(0) \cdot \frac{P_{i,j}}{Q_{i,j}},\]
where
\[ Q_{i,j} = \prod_{\substack{0\leq j'\leq j\\j'< i'\leq i}}\left( 1 - \frac{(i'-j')\cdot t}{(n-k+i')(m-j')}\right)\]  
and $P_{i,j}$ is a polynomial in $t$ with constant term 1. 
After the substitution $t = (mn/k)\cdot s$, the equation can be written
\[ Q_{i,j} = \prod_{\substack{0\leq j'\leq j\\j'< i'\leq i}}\left( 1 - \frac{(i'-j')\cdot mns}{k(n-k+i')(m-j')}\right)\]  
Here $Q_{i,j}$ depends implicitly on $k$, $m$, and $n$, but fixing $k$ and $n$, we can let $m$ tend to infinity and cancel the factors $m-j'$ against factors $m$. Therefore the limit 
\begin{equation} \label{Qlimit} 
\lim_{m\to\infty} Q_{i,j}(s) = \prod_{\substack{0\leq j'\leq j\\j'< i'\leq i}}\left( 1 - \frac{(i'-j')ns}{k(n-k+i')}\right)
\end{equation}
holds coefficientwise as polynomials in $s$.

We claim that in the large $m$ limit, the numerators $P_{i,j}$ have a similar behavior with the product taken over $j'>0$, excluding the cases with $j'=0$, so that 
\begin{equation} \label{Plimit} 
\lim_{m\to\infty} P_{i,j}(s) = \prod_{\substack{1\leq j'\leq j\\j'< i'\leq i}}\left( 1 - \frac{(i'-j')ns}{k(n-k+i')}\right)
\end{equation}
and moreover, 
\begin{equation} \label{binomial} 
L_{i,j}(0) \sim \frac{\binom{i}{j}}{m^{i-j}}.
\end{equation}

In order to verify this inductively from the recursion \eqref{recursionL(t)}, we rewrite the recursion in terms of $s$, and again replace factors of $m-j$ by $m$.
\begin{multline} \label{recursionL(s)}
 L_{i,j} \sim \left(1-\frac{(i-j)ns}{k(n-k+i)}\right)^{-1} \\
 \cdot \left( \frac{L_{i-1,j}}{m} + \frac{(i-j+1)mns}{k(n-k+i)}L_{i,j-1} + L_{i-1,j-1}\right).
 \end{multline}

We want to rewrite the right-hand side of \eqref{recursionL(s)} with a denominator of \eqref{Qlimit}, assuming that \eqref{Plimit} and \eqref{binomial} hold whenever $(i, j)$ is replaced by $(i-1,j)$, $(i,j-1)$ or $(i-1,j-1)$.
The factor 
\[ 1-\frac{(i-j)ns}{k(n-k+i)} \]
in the denominator comes directly from the recursion. Rewriting the rest of the right-hand side of \eqref{recursionL(s)} using the induction hypothesis, we can put the three terms on a common denominator and also extract a factor $1/m^{i-j}$.
This gives a remaining numerator of 
\begin{multline}
\prod_{\substack{1\leq j'\leq j \\ j'< i'\leq i \\ (i',j') \neq (i,j)}}  \left( 1 - \frac{(i'-j')ns}{k(n-k+i')} \right) \\
\cdot \Bigg(\binom{i-1}{j} \left(1-\frac{ins}{k(n-k+i)}\right) + \binom{i}{j-1}\frac{(i-j+1)ns}{k(n-k+i)} \\
+ \binom{i-1}{j-1} \left(1-\frac{ins}{k(n-k+i)}\right) \Bigg)
\end{multline}

Here the first product already involves all the factors it should, except the one corresponding to $i'=i$ and $j'=j$. The second factor simplifies to
\[  
\binom{i}{j} \left(1-\frac{ins}{k(n-k+i)}\right) + \binom{i}{j-1}\frac{(i-j+1)ns}{k(n-k+i),}
\]
since 
\[ \binom{i-1}{j} + \binom{i-1}{j-1} = \binom{i}{j},\]
and further to 
\begin{multline} \label{lastFactor}
\binom{i}{j} + \left( \binom{i}{j-1}(i-j+1) - \binom{i}{j}\cdot i \right) \cdot \frac{ns}{k(n-k+i)} \\
= \binom{i}{j} \left(1-\frac{(i-j)ns}{k(n-k+i)}\right),
\end{multline}
since 
\[ 
\binom{i}{j-1}(i-j+1)-\binom{i}{j}\cdot i = \frac{i!(j-i)}{j!(i-j)!} = -\binom{i}{j}(i-j).
\]
The factor \eqref{lastFactor} gives the binomial coefficient and the missing factor of \eqref{Plimit}, thereby completing the induction.

Taking $i=j=k$, we can informally write our result as
\begin{equation} \label{informal}
\lim_{m\to\infty} F_{k,m,n}(s) = \frac{ \prod_{1\leq j'<i'\leq k}\left( 1 - \frac{(i'-j')\cdot ns}{k(n-k+i')}\right) }{ \prod_{0\leq j'<i'\leq k}\left( 1 - \frac{(i'-j')\cdot ns}{k(n-k+i')}\right) }
\end{equation}
but this has to be interpreted in a specific way to make sense. The rescaling to the variable $s$ instead of $t$ is necessary, since otherwise there is no convergence. Moreover, the interpretation is that it is possible to write $F_{k,m,n}$ with a numerator and a denominator in such a way that the numerator is a polynomial in $s$ of degree $k(k-1)/2$, the denominator is a polynomial in $s$ of degree $k(k+1)/2$, and as $m$ tends to infinity with $k$ and $n$ fixed, each coefficient of these polynomials converges to the corresponding coefficient of the right-hand side of \eqref{informal}.

It looks as if we could just cancel all the factors of the numerator against the same factors of the denominator, but this is not the case. The actual numerators and denominators of $F_{k,m,n}(s)$ do not in general have any zeros in common, even though the zeros are close to each other.

At this point we invoke the well-known continuity of the complex zeros of a polynomial as a function of its coefficients. For fixed $k$ and $n$, if $m$ is large enough, the zeros of $F_{k,m,n}(s)$ will be close to the corresponding zeros of the right-hand side of \eqref{informal}, which are given by
\[ s = \frac{k(n-k+i')}{n(i'-j')} > \frac{k(n-k+i')}{n\cdot i'} \geq 1.\]
Therefore there are only finitely many values of $m$ for which any zero can be in the unit disk.

By symmetry we also conclude that for fixed $k$ and $m$, there are only finitely many values of $n$ for which $F_{k,m,n}(s)$ can have zeros in the unit disk.

Finally we consider the case that both $m$ and $n$ are large. Then we can simplify \eqref{informal} even further and conclude that the zeros of $F_{k,m,n}(s)$ will cluster around the points 
\[ s = \frac{k}{i'-j'}\] 
with one zero close to $k/(k-1)$, two zeros close to $k/(k-2)$, three near $k/(k-3)$ and so on, and finally $k-1$ zeros close to $k$. Therefore, for fixed $k$, if $m$ and $n$ are both large enough, $F_{k,m,n}(s)$ cannot have any zeros in the unit disk. This completes the proof of Proposition~\ref{P:finitely}. 

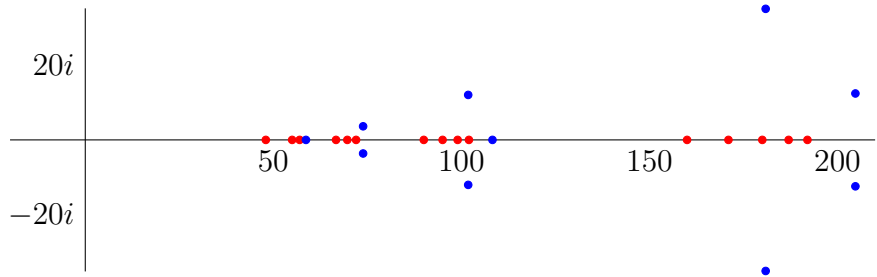
\begin{figure} [h]
\begin{center}
\begin{tikzpicture} [scale=0.05]

\draw (-20,0)--(210,0);
\draw  (0,-35)--(0,35);

\node[below] at (50,0) {50};
\node[below] at (100,0) {100};
\node[below] at (150,0) {150};
\node[below] at (200,0) {200};

\node[left] at (0,20) {$20i$};
\node[left] at (0,-20) {$-20i$};

\foreach \x in 
{48., 55., 57., 66.6667, 69.6667, 72., 90., 95., 99., 102., 160., 171., 180., 187., 192.}
\filldraw[red] (\x,0) circle (1.0);

\foreach \x/\y in 
{58.6699/0, 108.26/0, 73.87/3.61627, 73.87/-3.61627, 101.809/-11.9679, 101.809/11.9679, 
180.894/-34.8732, 180.894/34.8732, 204.754/-12.3515, 204.754/12.3515}
\filldraw[blue] (\x,\y) circle (1.0);

\end{tikzpicture}

\caption{In blue, the $10 = \binom{5}{2}$ complex zeros of $F_{5,12,20}(t)$, the moment generating function of the minimum cost of a 5-matching in a 12 by 20 complete bipartite graph, and in red, its $15 = \binom{6}{2}$ poles.  The five largest poles constitute a cluster surrounded by four zeros, the next cluster of four poles are surrounded by three zeros, and so on.
}
\label{F:5-12-20}
\end{center}
\end{figure}

Figure~\ref{F:5-12-20} shows the zeros and poles of the function $F_{5,12,20}(t)$.
The distinct clusters of poles and zeros are clearly visible. 

Our argument does not offer any explicit bounds on $m$ and $n$ in the conceivable counterexamples to Conjecture~\ref{C:zeroFree}. The proof of Proposition~\ref{P:finitely} relies on the stability of the complex zeros of a polynomial as its coefficients vary, and even though the zeros are known to depend continuously on the coefficients, they can behave rather badly in some cases as is demonstrated by the infamous Wilkinson polynomial \cite{Wilkinson}.

\end{document}